\documentclass[11pt]{article}
\usepackage{amsmath,enumerate,amsfonts,color,amssymb,amsthm, verbatim}

\usepackage{hyperref}
\usepackage[normalem]{ulem}

\usepackage{tikz}

\def\ZZ{{\mathbb Z}}
\def\RR{{\mathbb R}}
\def\Sphere{{\mathbb S}}
\def\mcA{{\mycal A}}

\newcommand{\f}{\varphi}

\def\bS{{\mathbb S}}
\def\eps{{\varepsilon}}

\newtheorem{theorem} {\sc  Theorem\rm} [section]

\newtheorem{proposition} [theorem] {\sc  Proposition\rm}
\newtheorem{prop} [theorem] {\sc  Proposition\rm}

\newtheorem{remark}[theorem]{\sc  Remark\rm}
\newtheorem{example}[theorem]{\sc  Example\rm}

\def\nd{\noindent}

\newcounter{marnote}

\DeclareFontFamily{OT1}{rsfs}{}
\DeclareFontShape{OT1}{rsfs}{m}{n}{ <-7> rsfs5 <7-10> rsfs7 <10-> rsfs10}{}
\DeclareMathAlphabet{\mycal}{OT1}{rsfs}{m}{n}

\newcommand{\h}{\mathcal{H}}

\def\be{\begin{equation}}
\def\ee{\end{equation}}

\newcommand{\R}{\mathbb{R}}
\newcommand{\Z}{\mathbb{Z}}

\def\be{\begin{equation}}
\def\ee{\end{equation}}
\def\bea#1\eea{\begin{align}#1\end{align}}

\numberwithin{equation}{section}

\begin{document}
\title{On the uniqueness of minimisers of Ginzburg-Landau functionals}

\author{Radu Ignat\thanks{Institut de Math\'ematiques de Toulouse \& Institut Universitaire de France, UMR 5219, Universit\'e de Toulouse, CNRS, UPS
IMT, F-31062 Toulouse Cedex 9, France. Email: Radu.Ignat@math.univ-toulouse.fr
}~, Luc Nguyen\thanks{Mathematical Institute and St Edmund Hall, University of Oxford, Andrew Wiles Building, Radcliffe Observatory Quarter, Woodstock Road, Oxford OX2 6GG, United Kingdom. Email: luc.nguyen@maths.ox.ac.uk}~, Valeriy Slastikov\thanks{School of Mathematics, University of Bristol, University Walk, Bristol, BS8 1TW, United Kingdom. Email: Valeriy.Slastikov@bristol.ac.uk}~ and Arghir Zarnescu\thanks{IKERBASQUE, Basque Foundation for Science, Maria Diaz de Haro 3,
48013, Bilbao, Bizkaia, Spain.}\,  \thanks{BCAM,  Basque  Center  for  Applied  Mathematics,  Mazarredo  14,  E48009  Bilbao,  Bizkaia,  Spain.
(azarnescu@bcamath.org)}\, \thanks{``Simion Stoilow" Institute of the Romanian Academy, 21 Calea Grivi\c{t}ei, 010702 Bucharest, Romania.}}

\date{}

\maketitle
\begin{abstract}
We provide necessary and sufficient conditions for the uniqueness of minimisers of the Ginzburg-Landau functional for $\R^n$-valued maps under a suitable convexity assumption on the potential and for $H^{1/2} \cap L^\infty$ boundary data that is non-negative in a fixed direction $e\in \bS^{n-1}$. Furthermore, we show that, when minimisers are not unique, 
the set of minimisers is generated from any of its elements using appropriate orthogonal transformations of $\mathbb{R}^n$. We also prove corresponding results for harmonic maps.\\


\end{abstract}

\tableofcontents

\section{Introduction and main results}

We consider the following Ginzburg-Landau type energy functional
\be\label{def:energ}
E_\eps(u)=\int_\Omega \Big[\frac{1}{2}|\nabla u|^2+\frac{1}{2\eps^2}W(1 - |u|^2)\Big]\,dx,
\ee 
where $\eps>0$, $\Omega \subset \R^m$ ($m \geq 1$) is a bounded domain (i.e., open connected set) with smooth boundary $\partial \Omega$ and the potential $W\in C^1((-\infty,1],\R)$ satisfies 
\be\label{ass:W}
 W(0)=0,\, W(t)>0 \hbox{ for all } t\in (-\infty, 1]\setminus \{0\}, \, W \textrm{ is strictly convex}.
\ee
We investigate the minimisers of the energy $E_\eps$ over the following set
\be\label{def:mcA}
\mcA:=\{ u\in H^1(\Omega;\R^n):\,  u=u_{bd}\textrm{ on }\partial\Omega\}, \quad  n\geq 1,
\ee consisting of $H^1$ maps with a given 
boundary data (in the sense of $H^{1/2}$-trace on $\partial \Omega$):
$$u_{bd}\in H^{1/2}\cap L^\infty( \partial\Omega; \RR^n).$$ In particular, we are interested in the question of uniqueness (or its failure) for the minimisers of $E_\eps$ in $\mcA$ for all range of $\eps>0$. 

In the case $\eps>0$ is large enough and \eqref{ass:W} holds, the energy $E_\eps$ is strictly convex (see Remark \ref{rem:convex}); as $E_\eps$ is lower semicontinuous in the weak $H^1$-topology and coercive over $\mcA$, there is a unique critical point $u_\eps$ of $E_\eps$ over $\mcA$ and this critical point is the global minimiser of the energy $E_\eps$. The case when $\eps>0$ is ``not too large'' is much more complex and, in general, one has to impose additional assumptions on the potential $W$ or/and the boundary data to understand the  uniqueness of minimisers and its failure even in the limit $\eps  \rightarrow 0$ (i.e. the related case of minimising harmonic maps).

There exists a large literature using various methods that addresses the question of uniqueness of minimisers in the framework of Ginzburg-Landau type models and the related harmonic map problem. See e.g. B\'ethuel, Brezis and H\'elein \cite{BBH_book}, Mironescu \cite{Mironescu_symmetry}, Pacard and Rivi\`ere \cite{Pacard_Riviere}, Ye and Zhou \cite{YeZhou-NA96},
Farina and Mironescu \cite{FarinaMironescu-CVPDE13}, Millot and Pisante \cite{Mil-Pis}, Pisante \cite{Pisante_JFA}, 
Gustafson \cite{Gustafson}, 
J\"ager and Kaul \cite{JagerKaul-ManuM79, JagerKaul-MA79}, Sandier and Shafrir \cite{SandierShafrir-CVPDE94, SandierShafrir-JFA17} and the references therein.

The approach of this work is based on the method that {we} previously  successfully used in the investigation of  stability and minimality properties of the critical points in the context of Landau-de Gennes model of nematic liquid crystals \cite{DRSZ,INSZ3,INSZ_AnnIHP,INSZ_CVPDE}. We refer to this tool as the Hardy decomposition technique (see Lemma A.1. in \cite{INSZ3}). We show that under the global assumption \eqref{ass:W} on the potential $W$ (namely its convexity), and an additional  assumption on the boundary data (namely its non-negativity in a fixed direction), we can use Hardy decomposition method to  provide simple yet quite general proofs of uniqueness of minimisers of energy $E_\eps$ and characterise its failure. In our forthcoming paper {\cite{INSZ-GL}} we will provide sufficient conditions for uniqueness of minimisers  under less restrictive local conditions on the potential, but imposing more restrictive conditions on the boundary data.

We begin the exposition with providing a simple result on the uniqueness and symmetry of minimisers, which will be subsequently extended in a much more general setting.  Assume  $\Omega \subset \R^2$ is the unit disk, $u: \Omega \to \R^3$, and the boundary data carries a given winding number $k \in \ZZ \setminus \{0\}$ on $\partial \Omega$, namely\footnote{We note that $u_{bd}$ is non-negative in $e_3$-direction.}
\be
u_{bd}(\cos\varphi,\sin\varphi)=(\cos(k\varphi),\sin (k\varphi),0)\in \bS^{1}\times \{0\}\subset \R^3, \quad \forall \varphi\in [0,2\pi).
	\label{Eq:ubdDegk}
\ee  
By restricting $E_\eps$ to a subset of $\mcA$ (with $n=3$) consisting of suitable rotationally symmetric maps (see condition \eqref{rot_sym} below), it is not hard to see that $E_\eps$ admits critical points of the form
\be\label{rad-sol}
u_\eps(r\cos\varphi, r\sin \varphi):=f_\eps(r)(\cos(k\varphi),\sin (k\varphi),0) \pm g_\eps(r)(0,0,1), \, r\in (0,1), \, \varphi\in [0,2\pi),
\ee 
where the couple $(f_\eps, g_\eps)$ of radial profiles solves the system
\be\label{syst:ode}
\begin{cases}
-f_\eps''-\frac{1}{r} {f_\eps'}+\frac{k^2}{r^2}f_\eps &= \frac{1}{\eps^2} f_\eps \, W'(1-f_\eps^2-g_\eps^2)\\
-g_\eps''-\frac{1}{r} g_\eps' &=  \frac{1}{\eps^2}g_\eps \, W'(1 - f_\eps^2 - g_\eps^2)
\end{cases}
\quad \textrm{ in } \, (0,1), 
\ee
subject to the boundary conditions
\be
f_\eps(0)=0, f_\eps(1)=1, g_\eps'(0)=0, g_\eps(1)=0.
	\label{syst:bc}
\ee

Our first result is the following:

\begin{theorem}\label{cor:symmetry}
Let $\Omega=\{x\in \R^2\, :\,   |x| <1\}$ be the unit disk in $\R^2$, $u: \Omega \to \R^3$, $W\in C^1((-\infty, 1], \R)$ satisfy \eqref{ass:W}, and boundary data $u_{bd}$ be given by \eqref{Eq:ubdDegk} where $k\in\Z \setminus\{0\}$ is a given integer. Then the following conclusions hold.
  \begin{enumerate}
  \item For every $\eps > 0$, any minimiser $u_\eps$ of the energy $E_\eps$ in the set $\mcA$ (with $n=3$) has the representation \eqref{rad-sol} where $f_\eps > 0$, $g_\eps \geq 0$ in $(0,1)$ and $(f_\eps, g_\eps)$ solves \eqref{syst:ode}-\eqref{syst:bc}.
 \item There exists $\eps_k > 0$ such that, for every $\eps \in (0, \eps_k)$, \eqref{syst:ode}-\eqref{syst:bc} has a unique solution $(f_\eps, g_\eps)$ with $g_\eps > 0$, and for $\eps \in [ \eps_k, \infty)$, no such solution exists.
  
 \item For every $\eps > 0$, 
 there is exactly one solution $(\tilde f_\eps, \tilde g_\eps)$ of \eqref{syst:ode}-\eqref{syst:bc} satisfying $\tilde f_\eps > 0$, $\tilde g_\eps \equiv 0$. 
  
 \item  
 For every $\eps \in (0, \eps_k)$, $E_\eps$ has exactly two minimisers that are given by $(f_\eps, \pm g_\eps)$ via \eqref{rad-sol} with $g_\eps>0$, while the critical point of $E_\eps$ corresponding to $(\tilde f_\eps, \tilde g_\eps\equiv0)$  via \eqref{rad-sol} is unstable.
\end{enumerate}
\end{theorem}

\medskip

\begin{remark} \label{rem:ODEUniq} 
We will prove in Proposition \ref{prop:dicho} that the critical values $\eps_k$ in Theorem \ref{cor:symmetry} satisfy $\eps_k=\eps_{-k}$ for every $k\geq 1$ and the sequence $(\eps_{k})_{k\geq 1}$ is increasing. Moreover, for every $k\in \ZZ\setminus \{0\}$, the two minimisers of $E_\eps$ over $\mcA$ given by the the solution $(f_\eps, \pm g_\eps)$ of the system \eqref{syst:ode}-\eqref{syst:bc} with $g_\eps>0$  converge strongly in $H^2(\Omega)$ as $\eps\to 0$ to the two minimising harmonic maps (see e.g. \cite{BBH}):
$$(r, \varphi)\in (0,1)\times [0, 2\pi)\mapsto \left(\frac{2r^k}{1+r^{2k}} \cos (k\varphi), \frac{2r^k}{1+r^{2k}}\sin(k\varphi), \pm\frac{1-r^{2k}}{1+r^{2k}}\right)\in \bS^2.$$
\end{remark}

The situation presented above indicates that uniqueness of minimisers can be lost in a discrete way (two alternative minimisers). This is due to an additional degree of freedom in the target space and becomes evident in a more general setting, when we work with $\RR^n$-valued maps: {if an isometry of $\RR^n$ does not change the boundary data, then it transforms a minimiser into a minimiser}. This is an important feature of our results that will be presented in the general setting for arbitrary dimensions $m \geq 1$ and $n \geq 1$ of base and target spaces, respectively. The main restriction in our treatment is the assumption that the boundary data is non-negative in a (fixed) direction $e\in \bS^{n-1}$, i.e.,
\begin{equation}\label{ass:bd}
u_{bd} \cdot {e} \geq 0 \quad \h^{m-1}\text{-a.e. in } \partial \Omega.
\end{equation}

Before formulating our main results we need to provide some basic definitions. For the rest of the paper we adopt the following notation: for any integrable $\R^n$-valued map $u$ on a measurable set $\omega$ we denote the essential image of $u$ on $\omega$ by
\begin{equation}
u(\omega) = \{u(x): x \in \omega \textrm{ is a Lebesgue point of } u\}.
	\label{def_image}
\end{equation}
It is clear that for a continuous function $u$ the above definition coincides with the (standard) range of $u$ on $\omega$.
We also define $Span \, u(\omega)$ to be the smallest vector subspace of $\R^n$ containing 
$u(\omega)$. 

The following result establishes the minimising property and uniqueness (up to { certain isometries}) for {\it critical points} of $E_\eps$ that are \emph{positive} in a fixed direction $e\in \bS^{n-1}$ inside the domain $\Omega\subset \R^m$.

\begin{theorem}\label{thm:main}
Let $m, n \geq 1$, $\Omega\subset\R^m$ be a bounded domain with smooth boundary, potential $W\in C^1((-\infty,1],\R)$ satisfying \eqref{ass:W}, and a boundary data $u_{bd}\in H^{{1}/{2}}\cap L^\infty( \partial\Omega; \RR^n)$ satisfying \eqref{ass:bd} in a fixed direction $e\in \bS^{n-1}$. Fix any $\eps>0$ and define $u_\eps \in H^1\cap L^\infty(\Omega,\R^n)$ to be a critical point  of the energy $E_\eps$ in the set $\mcA$, that is positive in the direction $e$ inside $\Omega$:
\be\label{ass:Phi}
u_\eps \cdot {e}>0\textrm{ a.e. in }\Omega.
\ee 
Then $u_\eps$ is a minimiser of $E_\eps$ in $\mcA$ and we have the following dichotomy:

\begin{enumerate}
\item If there exists a Lebesgue point $x_0\in\partial\Omega$ of $u_{bd}$ such that $u_{bd}(x_0)\cdot {e}>0$ then $u_\eps$ is the unique minimiser of $E_\eps$ in the set $\mcA$.

\item If $u_{bd}(x)\cdot {e}=0$ for $\h^{m-1}$-a.e. $x\in \partial\Omega$, then all minimisers of  $E_\eps$ in $\mcA$ are given by $Ru_\eps$ where $R\in O(n)$ is an orthogonal transformation of $\R^n$ satisfying $Rx=x$ for all $x\in  Span \, u_{bd}(\partial\Omega)$ (within the definition \eqref{def_image}).
\end{enumerate}
\end{theorem}

\bigskip
Using the above theorem, we prove the following result which completely characterises uniqueness and its failure for minimisers of the energy $E_\eps$. Note that the positivity assumption \eqref{ass:Phi} in a direction $e$ will not be enforced anymore on minimisers inside $\Omega$; however the non-negativity \eqref{ass:bd} of the boundary data $u_{bd}$ is fundamental.

\begin{theorem}\label{thm:main_inv}
Let $m$, $n$, $\Omega$, $W$, $u_{bd}$ be as in Theorem~\ref{thm:main}, $e\in \mathbb{S}^{n-1}$ be such that \eqref{ass:bd} holds and $V \equiv Span \, u_{bd}(\partial\Omega)$. Then for every $\eps>0$ there exists a minimiser $u_\eps$ of the energy  
$E_\eps$ on the set $\mcA$ and this minimiser is unique unless both following conditions hold:
\begin{enumerate}
\item $u_{bd}(x)\cdot {e}=0$ $\h^{m-1}$-a.e. $x\in \partial\Omega$,

\item the functional $E_\eps$ restricted to the set
\[
\mcA_{res} := \{u \in \mcA: u(x) \in Span(V \cup\{{e}\}) \text{ a.e.  in } \Omega\}
\]
has a minimiser $\tilde u_\eps$ with $\tilde u_\eps(\Omega) \not\subset V$. \footnote{We will see that any minimiser $\tilde u_\eps$ belongs to $C^1(\Omega)$, therefore, $\tilde u_\eps(\Omega)$ has the standard meaning. Moreover, if $\tilde u_\eps(\Omega) \not\subset V$, then the minimiser $\tilde u_\eps$ of $E_\eps$ over $\mcA_{res}$ is unique up to the reflection $\tilde u_\eps\cdot e\mapsto -\tilde u_\eps\cdot e$ (that keeps $u_{bd}$ unchanged due to point 1.).}
\end{enumerate}
Moreover, if uniqueness of minimisers of $E_\eps$ in $\mcA$ does not hold, then all minimisers of  $E_\eps$ in $\mcA$ are given by $Ru_\eps$ where $R\in O(n)$ is an orthogonal transformation of $\R^n$ satisfying $Rx=x$ for all $x\in V$.
\end{theorem}

Loosely speaking, the above result asserts that if $u_\eps$ is a minimiser of $E_\eps$ in $\mcA$ then 
\begin{itemize}
\item either $u_\eps(\Omega) \subset Span\, u_{bd}(\partial \Omega)$ and $u_\eps$ is the unique minimiser of $E_\eps$ in $\mcA$ (a phenomenon that we call ``non-escaping" combined with ``uniqueness" of minimisers, see the discussion below);
\item or $u_\eps(\Omega) \not\subset Span\, u_{bd}(\partial \Omega)$, and all minimisers of $E_\eps$ in $\mcA$ differ from $u_\eps$ by an orthogonal transformation which fixes every point of $Span\, u_{bd}(\partial \Omega)$ (that we call ``escaping" phenomenon combined with ``multiplicity"  of minimisers, see below).
\end{itemize}

\medskip

\noindent {\bf Harmonic map problem}. We note that Theorem \ref{thm:main_inv} is similar to a well-known result of Sandier and Shafrir \cite{SandierShafrir-CVPDE94} on the uniqueness of minimising harmonic maps into a closed hemisphere. In fact, our proof of Theorem \ref{thm:main_inv} can be adapted to give an alternative proof of their result. Our argument does not assume the smoothness of boundary data and does not use the regularity theory of minimising harmonic maps, which appears to play a role in the argument of \cite{SandierShafrir-CVPDE94}.

\begin{theorem}\label{prop:HMP_inv}
Let $m\geq 1$, $n \geq 2$, $\Omega\subset\R^m$ be a bounded domain with smooth boundary, $u_{bd} \in H^{1/2}(\partial\Omega;\mathbb{S}^{n-1})$ be a boundary data satisfying \eqref{ass:bd} in a direction $e\in \bS^{n-1}$, and $V \equiv Span\, u_{bd}(\partial\Omega)$. Then there exists a minimising harmonic map $u \in \mcA \cap H^1(\Omega;\mathbb{S}^{n-1}) $ and this minimising harmonic map is unique unless both following conditions hold:
\begin{enumerate}
\item $u_{bd}(x)\cdot {e}=0$ $\h^{m-1}$-a.e. $x\in \partial\Omega$,

\item  the Dirichlet energy $E(w) = \frac{1}{2}\int_\Omega |\nabla w|^2\,dx$ restricted to the set
\[
\mcA_{res}^* := \{w \in \mcA \cap H^1(\Omega;\mathbb{S}^{n-1}): w(x) \in Span(V \cup\{{e}\}) \text{ a.e.  in } \Omega\}
\]
has a minimiser $\tilde u$ with $\tilde u(\Omega) \not\subset V$.
\end{enumerate}
Moreover, if $u$ is not the unique minimising harmonic map in $\mcA \cap H^1(\Omega;\mathbb{S}^{n-1})$, then all minimising harmonic maps in $\mcA \cap H^1(\Omega;\mathbb{S}^{n-1})$ are given by $Ru$ where $R\in O(n)$ is an orthogonal transformation of $\R^n$ satisfying $Rx=x$ for all $x\in V$.
\end{theorem}

Similar to the argument for Theorem \ref{thm:main_inv}, the proof of  Theorem \ref{prop:HMP_inv} uses an analogue of Theorem \ref{thm:main} for weakly harmonic maps positive in a given direction $e$ inside $\Omega$; this  result is given in Theorem \ref{prop:HMP}.

\bigskip

\noindent {\bf Relation between uniqueness, stability, and ``escaping" phenomenon}. 
We would like to point out some implications of Theorems~\ref{thm:main_inv} and \ref{prop:HMP_inv} when the boundary data $u_{bd}$ {\it vanishes in a direction $e$ at the boundary}, i.e.,  ${\rm codim}_{\R^n} Span \, u_{bd}(\partial \Omega)\geq 1$. In this case we have the following observations:

\smallskip

{\bf i)} The uniqueness of minimisers in Theorems \ref{thm:main_inv} and \ref{prop:HMP_inv} is equivalent to the confinement of the range $u(\Omega)$ of minimisers $u$ of $E_\eps$ (resp. of $E$) inside the vector space $V=Span\, u_{bd}(\partial \Omega)$, which we will refer to as the ``non-escaping" phenomenon.
$$
\textrm{\bf Uniqueness of minimiser} \quad  \Longleftrightarrow \quad \textrm{\bf ``Non-escaping" phenomenon for minimiser}.
$$

The opposite situation when the range $u(\Omega)$ is not confined inside the vector space $V$ (i.e., $u_\eps(\Omega) \not\subset Span\, u_{bd}(\partial \Omega)$) - that we call ``escaping" phenomenon - often occurs in physical models (e.g. micromagnetic vortices, liquid crystal defects):  in order to lower the energy, the minimisers prefer to escape along a direction $e$ transversal to the natural vector space $V$ where boundary data $u_{bd}$ lives. This ``escape''  is usually enforced by topological constraints in the boundary data $u_{bd}$.  

We note that the ``escaping" phenomenon dictates the loss of uniqueness of a minimiser in the following way: the minimiser $\tilde u_\eps$ of $E_\eps$ over $\mcA_{res}$ (resp. $\tilde u$ of $E$ over $\mcA_{res}^*$) is unique up to the reflection {about the vector space orthogonal to $e$} (see Remark~\ref{rem:nonescape}). All the other minimisers (if any) of $E_\eps$ over $\mcA$ (resp. of $E$ over $\mcA\cap H^1(\Omega; \mathbb{S}^{n-1})$) are obtained by orthogonal transformations of $\tilde u_\eps$ (resp. $\tilde u$) fixing all points of $V$.

\medskip

{\bf ii)} The ``escaping" phenomenon is closely related to stability properties of critical points. In particular, in Theorem \ref{thm:main} (resp. Theorem \ref{prop:HMP}) we show that every ``escaping" critical point $u$ (i.e., \eqref{ass:Phi} holds inside $\Omega$) of $E_\eps$ over $\mcA$ (resp. of $E$ over $\mcA\cap H^1(\Omega; \mathbb{S}^{n-1})$) is in fact a minimiser and there are multiple minimisers as one can reflect $u$ about the orthogonal space to the escaping direction. On the contrary, in Propositions \ref{Lem:InPlane+Stab=>Min} and \ref{Lem:HPInPlane+Stab=>Min} we show that if a critical point of $E_\eps$ over $\mcA$ (resp. of $E$ over $\mcA\cap H^1(\Omega; \mathbb{S}^{n-1})$) does not ``escape" then its stability is equivalent with its minimality and by point {\bf i)} it is the unique minimiser. 

Therefore, we have the following dichotomy in the behaviour of minimisers: 
\begin{itemize}
\item either {\bf ``escaping" of minimisers}:  

\noindent $\{\textrm{\bf ``escaping critical points}\}=\{\textrm{\bf minimisers}\}$ with at least two elements;

\item  or {\bf ``non-escaping" of minimisers}:  

\noindent $\{\textrm{\bf ``Non-escaping" stable critical point}\}= \{\textrm{\bf minimiser}\}$ with only one element.
\end{itemize}
As showed at Theorem \ref{cor:symmetry}, point 3., one can have at the same time ``escaping" minimisers and ``non-escaping" {\it unstable} critical points. 

\medskip

{\bf iii)} In the case of the harmonic map problem, if $V$ has dimension $m-1$ and the boundary data $u_{bd}$ is smooth and carries a {\it nontrivial topological degree} as a map $u_{bd}:\partial \Omega\to V\cap \mathbb{S}^{n-1}$ with $n>m\geq 2$, then \footnote{The fact that ``escaping phenomenon" implies smoothness is a consequence of the boundary regularity result of Schoen and Uhlenbeck \cite{SU-Bdry} and the regularity result of J\"ager and Kaul \cite{JagerKaul-ManuM79} for maps with values into $\bS^{n-1}\cap\{u_{n}>\delta\}$ for some $\delta>0$. The other implication is obvious arguing by contraposition: ``non-escaping" smooth maps $u:\Omega\to V\cap \mathbb{S}^{n-1}\simeq \mathbb{S}^{m-1}$ have zero Jacobian determinant as the partial derivatives $\partial_1 u, \dots, \partial_m u$ are linearly dependant in the tangent space $T_u \mathbb{S}^{m-1}$ and therefore, $u$ carries a zero topological degree at the boundary $\partial \Omega$. } 
the following holds for minimising harmonic maps in $\mcA\cap H^1(\Omega; \mathbb{S}^{n-1})$: 
$$\textrm{\bf ``Escaping" phenomenon for minimisers} \quad  \Longleftrightarrow \quad \textrm{\bf Smoothness for minimisers}.$$

We illustrate the above points on a natural and useful example.
\begin{example}
\label{examp}
Let $\Omega=B^m$ be the unit ball in $\R^m$, $m\geq 2$, $n\geq m+1$ and boundary data $u_{bd}(x)=(\frac{x}{|x|}, 0_{\R^{n-m}})$ for every $x\in \partial \Omega= \bS^{m-1}$. It is known (see J\"ager and Kaul \cite{JagKaul}) that the equator map $$u(x)=(\frac{x}{|x|}, 0_{\R^{n-m}}) \quad \textrm{ for every } \,  x\in B^{m}$$ is a minimising $\mathbb{S}^{n-1}$-valued harmonic map with the boundary data $u_{bd}$ {\bf if and only if} $m\geq 7$. \footnote{However, it is a minimising map among all $\Sphere^{m-1}$-valued maps for $m \geq 3$ (see Brezis, Coron and Lieb \cite{BrezisCoronLieb} and Lin \cite{Lin-CR87}), where we identify $\Sphere^{m-1}=\Sphere^{m-1}\times \{0_{\R^{n-m}}\}$.} 

Let us now explain how our results recover the above statement.
We start by illustrating {\bf ii)} of the above discussion: the stability of the equator map (with respect to $\Sphere^{n-1}$-valued maps) is equivalent to the validity of the following inequality
(see \eqref{numar_nou})
\[
\int_{B^m} |\nabla \varphi|^2\,dx \geq (m-1)\int_{B^m} \frac{|\varphi|^2}{|x|^2}\,dx \text{ for all } \varphi \in H^1_0(B^m)
\]
(because $|\nabla u|^2=(m-1)/|x|^2$).
In view of the sharp Hardy inequality in $m$ dimensions, this is true if and only if
\[
\frac{(m-2)^2}{4} \geq m - 1 \  \Longleftrightarrow \  m \geq 7.
\]
Invoking Proposition \ref{Lem:HPInPlane+Stab=>Min}, we can recover the aforementioned result of J\"ager and Kaul \cite{JagKaul}, i.e., $u$ is a ``non-escaping" stable critical point  and hence the unique minimising harmonic map in $\mcA$ if  and only if $m\geq 7$. Therefore, the escaping phenomenon occurs in dimensions $2 \leq m \leq 6$ when $\mathbb{S}^{n-1}$-valued  minimising harmonic map are smooth (see {\bf iii)}), while the ``non-escaping" phenomenon occurs in dimensions $m \geq 7$ when the equator map is the unique minimising harmonic map (see {\bf i)}).

\end{example}

Inspired by Example \ref{examp}, we will prove that the ``non-escaping" phenomenon also occurs in the Ginzburg-Landau problem for every $\eps>0$ in dimensions $n>m \geq 7$ when the equator map is the boundary data
$$u_{bd}(x)=(x, 0_{\R^{n-m}}), \quad x\in \partial B^m.$$
In that situation, the unique minimiser is the ``non-escaping" radial critical point of $E_\eps$ given by
\be
\label{def_sol_equa}
u_\eps(x)=(h_\eps(|x|)\frac{x}{|x|}, 0_{\R^{n-m}}), \quad \textrm{ for all } x\in B^m,\ee
where the radial profile $h_\eps$ is the increasing positive solution of 
\be
\label{1}
\left\{\begin{array}{l}
-h''_\eps-\frac{m-1}{r}h'_\eps+\frac{m-1}{r^2}h_\eps=\frac{h_\eps}{\eps^2} W'(1-h_\eps^2) \quad \textrm{for every } r\in(0,1),\\
h_\eps(0)=0, h_\eps(1)=1
\end{array}
\right.
\ee
(see e.g. \cite{ODE_INSZ}).

\medskip

\begin{theorem}
\label{GL_equator}
Let $W\in C^1((-\infty,1],\R)$ be a potential satisfying \eqref{ass:W}. If $n>m\geq 7$, then for every $\eps>0$, $u_\eps$ given in \eqref{def_sol_equa} is the unique minimiser of $E_\eps$ under the boundary data $u_{bd}=(x, 0_{\R^{n-m}})$ for $x\in \partial B^m$.
\end{theorem}

\begin{remark}
i) Within the hypotheses of Theorem \ref{GL_equator}, a consequence of Theorems \ref{thm:main} and \ref{GL_equator}  is the non-existence of ``escaping" critical points of $E_\eps$
for every $\eps>0$ if $n>m\geq 7$.

ii) When $2\leq m\leq 6$ and $n>m$, the non-existence of ``escaping" critical points of $E_\eps$ still holds
provided that $\eps$ is large (in particular, if $E_\eps$ is strictly convex, then $u_\eps$ in \eqref{def_sol_equa} is the only critical point of $E_\eps$ with the boundary data $u_{bd}$). However, for small $\eps>0$, the minimisers of $E_\eps$ with the boundary data $u_{bd}$ will ``escape" from $\R^m\times\{0_{\R^{n-m}}\}$ 
(while the ``non-escaping" critical point $u_\eps$ in \eqref{def_sol_equa} becomes unstable). This is because any minimiser of $E_\eps$ converges in $H^1(B^m)$ as $\eps\to 0$ to a minimising harmonic map $u_*$ (see e.g. \cite{BBH}) and by Example \ref{examp} every minimising harmonic map $u_*$ escapes from the subspace $\R^m\times\{0_{\R^{n-m}}\}=Span \, u_{bd}(\partial \Omega)$.
\end{remark}

The paper is organized as follows. In Sections~\ref{sec:thm_main} and \ref{sec:thm_inv}, we prove Theorem~\ref{thm:main} and Theorems~\ref{thm:main_inv} and \ref{GL_equator}, respectively. 
In section~\ref{sec:GL} we investigate the radially symmetric Ginzburg-Landau problem in its simplest form ($B^2 \to \R^3$) and prove the uniqueness and symmetry of minimisers formulated in Theorem~\ref{cor:symmetry}. In Section~\ref{sec:HMP} we investigate the harmonic map problem and prove Theorem~\ref{prop:HMP}  and Theorem~\ref{prop:HMP_inv} -- the harmonic maps analogues of Theorem~\ref{thm:main} and Theorem~\ref{thm:main_inv}. Finally, in the Appendix~\ref{sec:d0} we discuss the ``escape" phenomenon for minimisers of $E$ and $E_\eps$ when the boundary data $u_{bd}$ carries a zero topological degree on $\partial \Omega$.

\section{Uniqueness of critical points of $E_\eps$. Proof of Theorem~\ref{thm:main}.}
\label{sec:thm_main}

Let us fix $\eps >0$ and $u_{bd}\in H^{1/2}\cap L^\infty(\partial \Omega; \R^n)$. It is clear that $\mcA\cap L^\infty(\Omega; \R^n)\neq \emptyset$ as it contains the harmonic extension of $u_{bd}$ inside $\Omega$. 
Assume that $u_\eps \in H^1\cap L^\infty(\Omega; \R^n)$ is a critical point of the energy $E_\eps$ in the set $\mcA$; then $u_\eps$ satisfies the Euler-Lagrange equations of $E_\eps$ in the weak sense, i.e.
\be
\label{E-L}
\int_{\Omega} \Big[\nabla u_\eps\cdot\nabla v-\frac{1}{\eps^2}W'(1-|u_\eps|^2)u_\eps v\Big]\,dx=0 \qquad \text{ for all } v \in H^1_0(\Omega;\R^n),
\ee
$$-\Delta u_\eps=\frac1{\eps^2} u_\eps \, W'(1-|u_\eps|^2)\quad \textrm{distributionally in } \, \Omega.$$
Note that \eqref{E-L} makes sense since $W'(1-|u_\eps|^2)\in L^\infty(\Omega)$ (because $u_\eps\in L^\infty(\Omega)$ and $W\in C^1((-\infty, 1])$).
Moreover, standard interior elliptic regularity implies $u_\eps \in C^1(\Omega)$. 

\begin{proof}[Proof of Theorem~\ref{thm:main}] For simplicity of the presentation we drop dependence on $\eps$ in the indices below.
Also due to the invariance of the energy under the transformation $u(x)\mapsto Ru(x)$ for any $R\in O(n)$, we can assume that 
\be
\label{en}
e:=e_n=(0,\dots, 0, 1)\in \R^n,
\ee
and work with a boundary data satisfying \eqref{ass:bd} in $e_n$ direction.

The proof will be done in three steps. In the first step we consider the difference between the energies of a critical point $u$ and an arbitrary competitor $u+v$ and show that this difference is controlled from below by some quadratic energy functional $F(v)$. In the second step, we employ the positivity of $u \cdot e_n$ in \eqref{ass:Phi} and apply the Hardy decomposition method in order to show that $F(v)\geq 0$, which proves in particular that $u$ is a minimiser of $E$. Finally, in the third step, we characterise the situations when this difference is zero and thus completely classify the possible cases of non-uniqueness of minimisers. 

\bigskip
\par\noindent{\it Step 1: Lower bound for energy difference.} 
For any  $v\in H^1_0(\Omega;\R^n)$ we have
\begin{align*}
E(u+v)-E(u)
	&=\int_\Omega\Big[ \nabla u\cdot\nabla v +\frac{1}{2}|\nabla v|^2\Big]\,dx\\
		&\qquad +\frac{1}{2\eps^2}\int_\Omega \Big[W(1-|u+v|^2)-W(1-|u|^2)\Big]\,dx.
\end{align*}
Using the strict convexity of $W$ (see \eqref{ass:W}), we have
\[
W(1-|u+v|^2)-W(1-|u|^2)\ge -W'(1-|u|^2)(|u+v|^2-|u|^2),
\]
where equality holds if and only if $|u+v|=|u|$. The last two relations imply that
\begin{align}
E(u+v)-E(u)
	&\ge \int_\Omega \Big[\nabla u\cdot\nabla v-\frac{1}{\eps^2} W'(1-|u|^2)u\cdot v\Big]\,dx\nonumber\\
		&\qquad\qquad +\int_\Omega \Big[ \frac{1}{2} |\nabla v|^2-\frac{1}{2\eps^2}W'(1-|u|^2)|v|^2\Big]dx,\label{rel:endif1}
\end{align}
with equality holding if and only if $|u(x)+v(x)|=|u(x)|$ a.e. $x\in\Omega$.
Moreover, by \eqref{E-L}, we obtain
\bea\label{rel:endif3}
E(u+v)-E(u) \ge  \int_\Omega \Big[ \frac 12 |\nabla v|^2-\frac{1}{2\eps^2} W'(1-|u|^2) |v|^2\Big]\,dx =: F(v)
\eea
for all  $v\in H^1_0(\Omega;\R^n)$.

\bigskip
\par\noindent{\it Step 2: Minimality of $u$ and the Hardy decomposition method.} Now we want to show the non-negativity of the  quadratic energy $F(v)$ defined in \eqref{rel:endif3}. Recall that $\Phi:=u \cdot e_n$ belongs to $C^1(\Omega)$ and $\Phi>0$ in $\Omega$ (by \eqref{ass:Phi}) and satisfies weakly the Euler-Lagrange equation in \eqref{E-L}:
\be\label{eq:phi}
-\Delta \Phi -\frac{1}{\eps^2}W'(1-|u|^2) \Phi =0 \quad \textrm{in } \, \Omega.
\ee
Now, for any smooth map $v \in C_c^\infty (\Omega; \R^n)$ with compact support in $\Omega$, we can define $\psi= \frac{v}{\Phi} \in H_0^1(\Omega; \R^n)\cap L^\infty(\Omega; \R^n)$ and we claim (by Lemma A.1. in \cite{INSZ3}):
\be
\label{desired30}
F(v)
	=\frac 12 \int_\Omega \Phi^2 |\nabla \psi|^2\,dx = \frac{1}{2} \int_\Omega \sum_{j= 1}^n \Big|\nabla v_j - v_j\frac{\nabla \Phi}{\Phi}\Big|^2\,dx\geq 0 \quad  \text{ for all } v \in C_c^\infty (\Omega; \R^n).
\ee
Indeed, for $1 \leq j \leq n$, integration by parts implies
\begin{align*}
\int_\Omega \frac{1}{\eps^2}W'(1-|u|^2) v_j^2\,dx
	&= \int_\Omega \bigg[\frac{1}{\eps^2}W'(1-|u|^2) \Phi\bigg] \cdot \bigg[\Phi\,\psi_j^2\bigg]\,dx\\
	&\stackrel{\eqref{eq:phi}}{=} \int_\Omega \nabla \Phi \cdot \nabla (\underbrace{\Phi\,\psi_j^2}_{=v_j \psi_j})\,dx= \int_\Omega \Big[-\Phi^2 |\nabla \psi_j|^2 + |\nabla v_j|^2\Big]\,dx
\end{align*}
yielding the desired expression \eqref{desired30}. By density of $C_c^\infty (\Omega; \R^n)$ in $H^1_0(\Omega;\RR^n)$ and the continuity of $F(\cdot)$ in strong $H^1$ topology (as $W'(1-|u|^2)\in L^\infty(\Omega)$), we have by \eqref{rel:endif3} and \eqref{desired30}
that for all $v \in H^1_0(\Omega;\RR^n)$ 
$$F(v)\geq \frac 12 \int_\omega  \sum_{j= 1}^n \Big|\nabla v_j - v_j\frac{\nabla \Phi}{\Phi}\Big|^2 \,dx\geq 0 \quad \textrm{ for any compact } \omega\subset \Omega$$
(recall that $(\nabla \Phi)/\Phi\in L^\infty(\omega)$ for every compact set $\omega\subset \Omega$). 
As a direct consequence of \eqref{rel:endif3}, we deduce that $u$ is a minimiser of $E$ over $\mcA$.
In order to determine when $F(v)=0$, we note that for $v\in H^1_0(\Omega)$ and  any open set $\omega\subset \Omega$ compactly supported in $\Omega$:
$$
0=F(v)\geq \frac 12 \int_\omega \Phi^2 |\nabla \big(\frac{v}{\Phi} \big)|^2\,dx\ge 0.$$
Thus we have $v(x) = \lambda_\omega \Phi(x)$, for a.e. $x\in\omega$ and some constant vector $\lambda_\omega \in \RR^n$. Since $\omega$ is arbitrary we have that there exists a $\lambda\in \RR^n$ such that $v(x) = \lambda \Phi(x)$, for a.e. $x\in\Omega$.

\bigskip
\par\noindent{\it Step 3: Characterisation of non-uniqueness of minimisers.} We now take another minimiser $w$ of $E$ over $\mcA$ and denote  $v:=w-u\in H^1_0$.
 From Steps 1 and 2 we know that $E(w)=E(v+u)=E( u)$ implies
 \begin{itemize}
 \item there is a pointwise equality $|v+u|=| u|$ a.e. in $\Omega$, 
 \item $v=\lambda \Phi$ for some $\lambda \in \R^n$ and the scalar function $\Phi=u \cdot e_n$.
\end{itemize}
\medskip 

\noindent {\it Situation 1. $u_{bd}(x_0)\cdot e_n>0$ at some Lebesgue point $x_0 \in \partial\Omega$ of $u_{bd}$}. We want to show uniqueness of minimisers, i.e., $\lambda=0$ meaning that $v=0$ in $\Omega$. Assume by contradiction that $\lambda_j\neq 0$ for some $1\leq j\leq n$. Then $\Phi=\frac{v_j}{\lambda_j}\in H^1_0(\Omega)$. Since the $n$-th component $u_{bd, n}$ of $u_{bd}$ is the $H^{1/2}$-trace of $\Phi$, we have $u_{bd, n}(x)=\frac{v_j(x)}{\lambda_j}=0$ for $\h^{m-1}$-a.e. $x\in \partial \Omega$. As $x_0$ is a Lebesgue point of $u_{bd}$ we would obtain $u_{bd, n}(x_0)=0$ which contradicts the assumption of Situation 1. Therefore, $\lambda=0$ and $u$ is the only minimiser of $E$ in $\mcA$. 

\medskip

\noindent {\it Situation 2. $u_{bd}(x_0)\cdot e_n=0$ for $\h^{m-1}$-a.e. $x\in \partial \Omega$}. As the energy $E$ is invariant under rotation, we may assume that 
$V=\R^k\times\{0_{\R^{n-k}}\}$ (for some $0\le k\le n-1$) is the smallest vector space containing the set $u_{bd} (\partial\Omega)$ defined by \eqref{def_image}. 
(In the case $n=1$, then $V=\{0\}$). Note that the map
\be
\label{transf_n}
\tilde u:=(u_1,\dots, u_k,0,\dots,0, \sqrt{u^2_{k+1}+\dots+u_n^2}) 
\ee
belongs to $\mcA$ and does not increase the energy, i.e., $E(u)\geq E(\tilde u)$. Hence, as $u$ is a minimiser of $E$ on $\mcA$, $\tilde u$ is also a minimiser. Therefore, up to interchanging $u$ and $\tilde u$, we may assume 
$$
\left\{\begin{array}{l}
u_{k+1}=\dots=u_{n-1}\equiv 0 \textrm{ in }\Omega\\
u_n=\Phi\stackrel{\eqref{ass:Phi}}{>}0\textrm{ in }\Omega.
\end{array}\right.
$$
The relation $|v+ u|=|u|$ a.e. in $\Omega$ implies $2u\cdot v+|v|^2=0$ a.e. in $\Omega$. Since $v=\lambda\Phi$ and $\Phi>0$ in $\Omega$ we obtain 
\be
\label{numarul}
2\lambda\cdot  u+\Phi |\lambda|^2=0 \hbox{ a.e. in } \Omega.
\ee
In the case $n=1$, then $\lambda=\lambda_n$, $u=u_n=\Phi$ and $(\lambda+1)^2=1$ yielding $w=(\lambda+1)u=\pm u$. If $n\geq 2$, we argue as follows: the $H^1(\Omega)$-functions $2\lambda\cdot  u$ and $-\Phi |\lambda|^2$ coincide, therefore they have the same $H^{1/2}$-trace on $\partial \Omega$. As our assumption states that $\Phi$ has vanishing trace on $\partial\Omega$, we deduce that 
$\lambda\cdot u_{bd}(x)=0$ for $\h^{m-1}$-a.e. $x\in \partial\Omega$, in particular, for every Lebesgue point $x$ of $u_{bd}$ on $\partial \Omega$. Since  $V=\R^k\times\{0_{\R^{n-k}}\}$ is the smallest vector space containing $u_{bd}(\partial\Omega)$, it means that there are $k$ Lebesgue points $x^1, \dots, x^k\in \partial \Omega$ such that $\{u_{bd}(x^j)\}_j$ is a basis of $V$ and $\lambda\cdot u_{bd}(x^j)=0$ for every $1\leq j\leq k$. It follows that $\lambda_1=\lambda_2=\dots=\lambda_k=0$ and therefore, recalling that $u_{k+1}=\dots=u_{n-1}\equiv 0 \textrm{ in }\Omega$, we have
by \eqref{numarul}:
\[
2\lambda_n\Phi+(\lambda_{k+1}^2+\dots+\lambda_n^2)\Phi=0\textrm{ a.e. in }\Omega.
\]
As $\Phi>0$ in $\Omega$, we obtain
\[
\lambda_{k+1}^2+\dots+\lambda_{n-1}^2+(\lambda_{n} +1)^2=1,
\]
hence we can find $R\in O(n)$ such that $R x= x$ for all $x \in V$ and $$Re_n=(0,\dots,0,\lambda_{k+1},\dots,\lambda_{n-1}, \lambda_n+1).$$ This implies $w=u+v=u+\lambda \Phi=Ru$ as required. The converse statement is obvious: if $u$ is a minimiser of $E$ over $\mcA$ and $R\in O(n)$ is a transformation fixing all points of $V$, then $Ru$ is also a minimiser of $E$ over $\mcA$ (because $E$ is invariant under isometries). 
\end{proof}

\section{Uniqueness of minimisers of $E_\eps$. Proof of Theorems~\ref{thm:main_inv} and \ref{GL_equator}}
\label{sec:thm_inv}

In this section we prove the uniqueness result of Theorem~\ref{thm:main_inv}.

\begin{proof}[Proof of Theorem~\ref{thm:main_inv}] For simplicity of the presentation we drop dependence on $\eps$ in the indices below. Also, by rotation invariance of the energy, we may assume that $e=e_n$ (as in \eqref{en}).
We divide the proof in several steps:

\medskip
\noindent{\it Step 1: Existence of a minimiser of $E$ over $\mcA$}. We apply the direct method in the calculus of variation: if $(u_k)_k$ is a minimising sequence for $E$ in $\mcA$, then $(\nabla u_k)_k$ is bounded in $L^2(\Omega)$. As $u_k=u_{bd}$ on $\partial \Omega$, by the Poincar\'e-Wirtinger inequality, we deduce that $(u_k)_k$ is bounded in $H^1(\Omega)$ and we conclude by using lower semi-continuity of $E$ with respect to the weak $H^1$-topology (as $W$ is a nonnegative continuous function).

\medskip 
 
 \noindent{\it Step 2: Any minimiser $u$ of $E$ over $\mcA$ belongs to $L^\infty(\Omega)$, namely}
\begin{equation}
|u| \leq \max\{1, \sup_{\partial\Omega} |u_{bd}| \}=: \bar \rho \quad \text{ a.e. in } \Omega.
	\label{Eq:|u|Bnd}
\end{equation}
Indeed, we start by defining the following functions in $\Omega$: $\rho=|u|\in H^1(\Omega; \R_+)$, 
$$
v=\begin{cases} u &\textrm{ if $\bar \rho\geq \rho$},\\
\frac{\bar \rho}{\rho} u &\textrm{ if $\bar \rho< \rho$} 
\end{cases}
$$
and $\eta:=|v|=\min(\rho, \bar \rho)$. Then obviously $|\nabla v|=|\nabla u|$ in the set $\{\bar \rho \geq \rho\}$. In the complementary set  $\tilde \Omega:=\{\bar \rho < \rho\}$, we claim that $|\nabla v|\leq |\nabla u|$ $\h^{m-1}$-a.e.; for that, we set $\psi:=\frac{u}{\rho}\in H^1(\tilde \Omega)$ (because $\rho > \bar \rho\geq 1$ inside $\tilde \Omega$). Starting from $u=\rho \psi$ in $\tilde \Omega$, we deduce that $\nabla u_j= \psi_j \nabla \rho + \rho \nabla \psi_j$ for every $1\leq j\leq n$. Since $\sum_{j=1}^n \psi_j^2=1$ in $\tilde \Omega$, we have
\begin{align*}
|\nabla u|^2& = \sum_{j=1}^n |\nabla u_j|^2 = \sum_{j=1}^n |\nabla \rho \psi_j + \rho \nabla \psi_j|^2  \\ &=\sum_{j=1}^n |\nabla \rho|^2 |\psi_j|^2 + \rho^2 |\nabla \psi_j|^2 + \frac12 \nabla (\rho^2) \cdot \nabla (\psi_j^2) \\
&= |\nabla \rho|^2 + \rho^2 |\nabla \psi|^2\geq \bar \rho^2 |\nabla \psi|^2=|\nabla v|^2 \quad \textrm{ in } \tilde \Omega.
\end{align*}
Now, note that $W'(t) < 0$ for $t < 0$ because $W$ is strictly convex and has a minimum at $t = 0$. Therefore, one has that $W(1-\eta^2)\leq W(1-\rho^2)$ with equality if and only if $\eta=\rho$. In particular, $E(u) \geq E(v)$.
As $u$ is a minimum of $E$ over $\mcA$ and $v \in \mcA$, it follows that the above inequalities become equalities, i.e., $\eta=\rho$ a.e. in $\Omega$ which yields 
\eqref{Eq:|u|Bnd}.

\medskip

\noindent{\it Step 3: We show that any minimiser $u$ of $E$ over $\mcA$ satisfies the following dichotomy
\begin{equation}
\text{either $|u_n| > 0$ or $u_n \equiv 0$ in $\Omega$, }
	\label{Eq:Phi>or=0}
\end{equation}
where $u_n := u \cdot e_n$}. Indeed, as $u\in L^\infty$ by Step 2, interior elliptic regularity yields $u \in C^1(\Omega)$. Now, note that $$\hat u = (u_1, \dots,  u_{n-1}, |u_n|)$$ is also a minimiser of $E$ over $\mcA$ (here \eqref{ass:bd} is essential). In particular $|u_n|\in C^1(\Omega)$ and satisfies
\[
-\Delta |u_n| - \frac{1}{\eps^2}W'(1-|u|^2 ) |u_n| = 0 \text{ in } \Omega.
\]
As $|u_n| \geq 0$, we thus have by the strong maximum principle that either $|u_n| > 0$ or $|u_n| \equiv 0$ in $\Omega$. 

\medskip

\noindent{\it Step 4: Characterisation of (non-)uniqueness of minimisers}. Let $u$ be a minimiser of $E$ over $\mcA$. By Step 3, up to a sign change~\footnote{Due to \eqref{ass:bd}, the situation $u_n<0$ inside $\Omega$ can occur only if $u_{bd}\cdot e_n=0$ on $\partial \Omega$; in which case, the sign change $u_n\mapsto -u_n$ is compatible with the boundary data.}, we may assume that $u_n\geq 0$ in $\Omega$.

\smallskip

\noindent {\it Situation 1. $u_{bd}(x_0) \cdot e_n > 0$ at some Lebesgue point $x_0\in \partial \Omega$ of $u_{bd}$}. This assumption implies that $u_n$ cannot be identically $0$ in $\Omega$, therefore by Step 3, we must have $u_n > 0$ in $\Omega$ (i.e., \eqref{ass:Phi} holds). Then by point 1. in Theorem \ref{thm:main}, $u$ is the unique minimiser of $E$. 

\smallskip

\noindent {\it Situation 2. $u_{bd} \cdot e_n = 0$  in $\partial\Omega$}. As the functional $E$ is invariant under rotation, we may assume that $V = \RR^k \times \{0_{\R^{n-k}}\}$ (for some $0\leq k\leq n-1$) is the smallest vector subspace of $\RR^{n-1} \times \{0\}$ which contains $u_{bd}(\partial\Omega)$. (If $n=1$, then $V=\{0\}$). Using the transformation \eqref{transf_n} as in Step 3 of the proof of Theorem~\ref{thm:main}, we can assume without loss of generality that the given minimiser $u$ of $E$ satisfies $u_{k + 1} \equiv \ldots \equiv u_{n-1} \equiv 0$ and $u_n\geq 0$ in $\Omega$. In particular, we have $u \in \mcA_{res}$, and so $u$ is also a minimiser of $E$ in $\mcA_{res}$. We distinguish two cases:

\smallskip

\noindent {\it Case A: $u(\Omega) \not\subset V$} (i.e., the point 2. of Theorem \ref{thm:main_inv} holds) 
implies non-uniqueness. Note that $u_n > 0$ in $\Omega$ (by \eqref{Eq:Phi>or=0}). It is clear that $(u_1, \ldots, u_k, u_{k+1}=0, \ldots, u_{n-1} = 0, -u_n)$ is also a minimiser of $E$ in $\mcA$ (and $\mcA_{res}$), i.e. minimisers of $E$ in $\mcA$  (resp. $\mcA_{res}$) are not unique. Furthermore,  by Theorem \ref{thm:main}, all minimisers of $E$ in $\mcA$ are of the form $Ru$ for some $R \in O(n)$ such that $Rx = x$ for all $x \in V$.

\smallskip

\noindent {\it Case B: $Span(u(\Omega)) = V$} (i.e., the point 2. of Theorem \ref{thm:main_inv} does not hold) implies uniqueness. Note that every other minimiser $v$ of $E$ over $\mcA_{res}$ 
also satisfies $Span(v(\Omega)) = V$ (because otherwise, $Span(v(\Omega)) \not\subset V$ for some minimiser $v$ and then Case A would imply that $u=Rv$ for some isometry $R$ fixing the space $V$ which contradicts $Span(u(\Omega)) = V$). In other words, every minimiser $u$ of $E$ over $\mcA_{res}$ has the form 
\begin{equation}\label{eq:minform}
u=(u_1, \dots, u_k, 0, \dots, 0) \quad \textrm{ in } \quad \Omega.
\end{equation}
Then we want to show that $u$ is the unique minimiser of $E$ in $\mcA$. In the case $V=\{0\}$ (in particular, for $n=1$), then clearly $0$ is the unique minimiser of $E$. Therefore, in the sequel, we may assume that $V\neq \{0\}$, in particular, $n\geq 2$ and $k\geq 1$. For that, note first that as $u$ is a minimiser of $E$ in $A_{res}$, we have by our assumption that $u_n \equiv 0$ in $\Omega$, so we cannot directly apply Theorem \ref{thm:main}. We will argue by contradiction: assume that $w = u + v$ is a different minimiser of $E$ in $\mcA$ where $0 \not\equiv v \in H^1_0(\Omega, \RR^n)$. One can easily check that in this case, $u\neq w$ implies that the map  $\tilde w$ obtained out of $w$ as in \eqref{transf_n} satisfies $\tilde w\neq u$ and $\tilde w$ is a minimiser of $E$ over $\mcA_{res}$. Therefore, by our assumption, without loss of generality, we assume that the components of index $k+1,\dots, n$ of $w$ (and hence of $v$) are zero. We will construct a minimiser $\xi$ of $E$ on $\mcA_{res}$ that has a positive component along $e_n$ which contradicts \eqref{eq:minform}. 

\smallskip

\noindent {\it Construction of the minimiser $\xi$}. Let $u,w$ and $v=w-u$ be as above. Recall that by \eqref{rel:endif3} in Step 1 of the proof of the Theorem~\ref{thm:main}, we have:
\begin{equation}
0 = E(u + v) - E(u) \geq \int_\Omega \Big[\frac{1}{2}|\nabla v|^2 - \frac{1}{2\eps^2} W'(1-|u|^2 )|v|^2\Big]\,dx=:F(v),
	\label{Eq:EqualEner}
\end{equation}
with equality if and only if $|u+v|=|u|$ a.e. in $\Omega$.
Note that we cannot use Hardy decomposition trick anymore to show the non-negativity of $F(v)$ as any component $u_1, \dots , u_k$ of $u$ could change
 sign. However, since $u$ is a minimiser of $E$ in $\mcA_{res}$, we have for any scalar test function $\varphi \in H^1_0(\Omega)$ that \footnote{Although $W$ is only $C^1$, one can easily check that $E(u + t\varphi e_n)$ is twice differentiable with respect to $t$ at $t = 0$ because $u\cdot e_n=0$ in $\Omega$.}
\be
\label{34}
0 \leq \frac{d^2}{dt^2}\Big|_{t = 0} E(u + t\varphi e_n) = \int_{\Omega} \big[|\nabla \varphi|^2 - \frac{1}{\eps^2} W'(1-|u|^2 ) \varphi^2\big]\,dx.
\ee
Taking $\varphi=v_j$ for $1 \leq j\leq n$ and summing over  all $j$ we obtain 
$$
 0 \stackrel{\eqref{Eq:EqualEner}}{\geq}  \int_\Omega \Big[\frac{1}{2}|\nabla v|^2 - \frac{1}{2\eps^2} W'(1-|u|^2 )|v|^2\Big]\,dx \stackrel{\eqref{34}}{\geq} 0.
$$
So equality holds in \eqref{Eq:EqualEner} yielding
\[
|u|=|u + v| = |w| \text{ in } \Omega.
\]
We are now ready to construct a competitor. We define
\[
\xi =  u + \frac{1}{2} v + \frac{1}{2}|v|e_n \in \mcA_{res}
\]
and note that $\xi_n=\frac12|v|>0$, i.e., $\xi_n \not\equiv 0$ (because $u\cdot e_n=v\cdot e_n=0$) yielding
\[
|\xi|^2 = | u+\frac{1}2 v|^2 + \frac14|v|^2=\frac12(|u|^2+| u + v|^2 ) = | u|^2 =|w|^2 \text{ in }  \Omega,
\]
\[
|\nabla \xi|^2 = |\nabla( u + \frac{1}{2}v)|^2 + \frac{1}{4}|\nabla |v||^2 \leq |\nabla(u + \frac{1}{2}v)|^2 + \frac{1}{4}|\nabla v|^2 = \frac{1}{2}\left( |\nabla u|^2 + |\nabla w|^2\right).
\]
It follows that
\[
E(\xi) \leq \frac{1}{2} \left(E(u) + E(w) \right) = E(w) = E(u),
\]
and so $\xi$ is also a minimiser of $E$ in $\mcA_{res}$. But since $\xi \cdot e_n \not\equiv0$, this is a contradiction with our preliminary observation that the range on $\Omega$ of all minimisers of $E$ over $\mcA_{res}$ stays inside $V$. This completes the proof. 
\end{proof}

\begin{remark} \label{rem:nonescape}
In the context of Theorem \ref{thm:main_inv}, if points $1$ and $2$ hold then the minimiser $\tilde u_\eps$ of $E_\eps$ over $\mcA_{res}$ is unique up to the reflection $\cal I$ mapping $\tilde u_\eps\cdot e\mapsto -\tilde u_\eps\cdot e$ w.r.t. $V$ inside $Span(V \cup
\{e\})$. Indeed, up to a rotation, we may assume that $e=e_n$ (as in \eqref{en}) and $V=\R^k\times\{0_{\R^{n-k}}\}$ for some $0\le k\le n-1$
and by \eqref{transf_n} and \eqref{Eq:Phi>or=0} we may assume that $\tilde u_\eps\cdot e_n> 0$ in $\Omega$; then point 2. in Theorem \ref{thm:main} yields the conclusion.   
\end{remark}

Using Theorem~\ref{thm:main_inv} we can prove the following result, which is of independent interest:
``non-escaping" stable critical points of $E_\eps$ are minimisers and moreover, they are unique. 

\begin{prop}\label{Lem:InPlane+Stab=>Min}
Let $m, n \geq 1$, $\Omega\subset\R^m$ be a bounded domain with smooth boundary. We consider a potential $W\in C^1((-\infty,1],\R)$ satisfying \eqref{ass:W} and a boundary data $u_{bd}\in H^{{1}/{2}}\cap L^\infty( \partial\Omega; e^\perp)$ for a direction $e\in \bS^{n-1}$, where
$$e^\perp:=\{v\in \R^n\,:\, v\cdot e=0\}.$$ 
For any fixed $\eps > 0$, if $u_\eps$ is a critical point of $E_\eps$ in $\mcA$, $u_\eps \in L^\infty(\Omega; e^\perp)$ and  is stable in direction $e$, i.e.
\[
\frac{d^2}{dt^2}\big|_{t=0}E_\eps(u_\eps+t\f e)=\int_\Omega \Big[|\nabla \varphi|^2 - \frac{1}{\eps^2}W'(1-|u_\eps|^2 )\, \varphi^2\Big]\,dx \geq 0 \text{ for all } \varphi \in H_0^1(\Omega),
\]
then $u_\eps$ is a minimiser of $E_\eps$ in $\mcA$. Moreover, if $u_\eps(\Omega)\subset Span \, u_{bd}(\partial \Omega)$, then $u_\eps$ is the unique minimiser of $E_\eps$ in $\mcA$.
\end{prop}

\begin{proof}[Proof of Proposition \ref{Lem:InPlane+Stab=>Min}]
Indeed, for any $v \in H_0^1(\Omega, \R^n)$ we have 
$E_\eps(u_\eps+v)-E_\eps(u_\eps)\geq F(v)$ thanks to \eqref{rel:endif3}. Employing the stability assumption in the direction $e$ and taking $\varphi=v_j$ for $1\leq j\leq n$ we obtain $F(v) \geq 0$ for every $v \in H_0^1(\Omega, \R^n)$, yielding that $u_\eps$ is a minimiser of $E_\eps$ over $\mcA$. (In the case $n=1$, then $e^\perp=\{0\}$ so that  
$u_{bd}=0$ on $\partial \Omega$ as well as $u_\eps=0$ in $\Omega$).
If $u_\eps(\Omega)\subset Span \, u_{bd}(\partial \Omega)$, then the uniqueness follows via Theorem \ref{thm:main_inv} (as the ``non-escaping" case holds, i.e., point $2$ fails).
\end{proof}

\begin{remark}
\label{rem:convex} For a given Lipschitz bounded domain $\Omega\subset \R^m$, under the assumption \eqref{ass:W} for the potential $W\in C^1((-\infty, 1])$, there exists $\eps_0>0$ (depending only on $W'(1)$ and $\Omega$) such that $E_\eps$ is strictly convex over $\mcA$ for every $\eps\geq \eps_0$ and every boundary data $u_{bd}\in H^{1/2}\cap L^\infty(\partial \Omega, \R^n)$; in this situation ($\eps\geq \eps_0$), there exists a unique critical point \footnote{This critical point is the global minimiser of $E_\eps$ over $\mcA$ as a minimiser always exists (see Step 1 in the proof of Theorem \ref{thm:main_inv}).} of $E_\eps$ over $\mcA$. Indeed, following the ideas in \cite{BBH_book} (see Theorem VIII.7), if $\lambda=\lambda_1(\Omega)$ is the first eigenvalue of $(-\Delta)$ on $\Omega$ (with zero Dirichlet data), then the functional $$G:v\mapsto \int_{\Omega} |\nabla v|^2-\lambda |v|^2\, dx$$ is convex \footnote{For $v, w\in H^1_0(\Omega)$ and 
$\alpha\in (0,1)$, we have $\alpha G(v)+(1-\alpha) G(w)-G(\alpha v+(1-\alpha)w)=\alpha(1-\alpha)G(v-w)\geq 0$ by the Poincar\'e inequality. } on $H^1_0(\Omega)$
which entails that $G$ is convex over $\mcA$ (as we can write $\mcA=u_0+H^1_0(\Omega, \R^n)$ for the harmonic extension $u_0\in \mcA$ of $u_{bd}$ inside $\Omega$). Moreover, by \eqref{ass:W} we deduce that $t\in [0, \infty) \mapsto \frac1{\eps^2}W(1-t)+\lambda t$ is strictly convex and {non-decreasing} for $\eps\geq \eps_0=
({|W'(1)|}/{\lambda})^{1/2}$, so that the functional 
$$u\in \mcA\mapsto \int_{\Omega}   \frac1{\eps^2}W(1-|u|^2)+\lambda |u|^2\, dx$$
is strictly convex.
\end{remark}

As a consequence of Proposition \ref{Lem:InPlane+Stab=>Min}, we prove Theorem \ref{GL_equator}:

\begin{proof}[Proof of Theorem \ref{GL_equator}]
By Proposition \ref{Lem:InPlane+Stab=>Min}, it is enough to prove the stability of $u_\eps$ given in \eqref{def_sol_equa}, i.e.,
$$F_\eps(v)=\frac12\int_{B^m} |\nabla v|^2-\frac1{\eps^2} W'(1-h_\eps^2)v^2\, dx\geq 0, \quad \forall v\in H^1_0({B^m}).$$
As a point has zero $H^1$ capacity in $\R^m$, by a standard density argument, it is enough to prove the above inequality for $v\in C^\infty_c({B^m}\setminus \{0\})$.
We consider the operator
$$L_\eps:=\nabla_{L^2}F_\eps=-\Delta-\frac1{\eps^2} W'(1-h_\eps^2).$$
Using the decomposition $v=h_\eps w$ for $v\in C^\infty_c({B^m}\setminus \{0\})$, we have (see e.g. \cite[Lemma~A.1]{INSZ3}):
\begin{align*}
2F_\eps(v)=\int_{B^m} L_\eps v\cdot v\, dx&=\int_{B^m} w^2 L_\eps h_\eps\cdot h_\eps\, dx+\int_{B^m} h_\eps^2 |\nabla w|^2\, dx\\
& =\int_{B^m} h_\eps^2 \bigg( |\nabla w|^2-\frac{m-1}{r^2}w^2\bigg)\, dx=:G_\eps(w),
\end{align*}
because \eqref{1} yields $L_\eps h_\eps\cdot h_\eps=-\frac{m-1}{r^2} h_\eps^2$ in ${B^m}$. To prove that $G_\eps(w)\geq 0$ for every $w\in C^\infty_c({B^m}\setminus \{0\})$, we use the decomposition $w=\f g$ with $\f=|x|^{-\frac{m-2}{2}}$ being the first eigenfunction of the Hardy's operator $-\Delta-\frac{(m-2)^2}{4|x|^2}$ in $\R^m\setminus\{0\}$ and $g\in C^\infty_c({B^m}\setminus \{0\})$. We compute
\begin{align*}
|\nabla w|^2=|\nabla \f|^2 g^2+|\nabla g|^2 \f^2+\frac12 \nabla (\f^2)\cdot \nabla (g^2).
\end{align*}
As $\f^2$ is harmonic in ${B^m}\setminus\{0\}$ and $|\nabla \f|^2=\frac{(m-2)^2}{4|x|^2}\f^2$, integration by parts yields
\begin{align*}
G_\eps(w)&=\int_{B^m} h_\eps^2 \bigg( |\nabla g|^2 \f^2+\frac{(m-2)^2}{4r^2}\f^2g^2-\frac{m-1}{r^2}\f^2 g^2\bigg)\, dx-\frac12 \int_{{B^m}} \nabla (\f^2)\cdot \nabla (h_\eps^2)g^2\, dx\\
&\geq \int_{B^m} h_\eps^2 |\nabla g|^2 \f^2+\bigg(\frac{(m-2)^2}{4}-(m-1)\bigg)\int_{{B^m}}\frac{h_\eps^2}{r^2}\f^2g^2\, dx\geq 0,
\end{align*}
where we have used $m\geq 7$ and $\frac12\nabla (\f^2)\cdot \nabla (h_\eps^2)=2\f \f' h_\eps h'_\eps\leq 0$ in ${B^m}\setminus \{0\}$.
\end{proof}

\section{Radially symmetric problem. Proof of Theorem \ref{cor:symmetry}}
\label{sec:GL}

Theorem \ref{cor:symmetry} is an immediate consequence of the following two results.

\begin{proposition}\label{prop:symmetry}
Under the hypotheses of Theorem \ref{cor:symmetry}, we have for every $\eps>0$:
  \begin{enumerate}
  \item Any minimiser $u_\eps$ of the energy $E_\eps$ in the set $\mcA$ (with $n=3$) has the representation \eqref{rad-sol} where $f_\eps > 0$, $g_\eps \geq 0$ in $(0,1)$ and $(f_\eps, g_\eps)$ solves \eqref{syst:ode}-\eqref{syst:bc}.
 
 \item There is at most one solution $(f_\eps, g_\eps)$ of \eqref{syst:ode}-\eqref{syst:bc} satisfying $g_\eps > 0$ and there is exactly one solution $(\tilde f_\eps, \tilde g_\eps)$ of \eqref{syst:ode}-\eqref{syst:bc} satisfying $\tilde f_\eps > 0$, $\tilde g_\eps \equiv 0$. 
  
 \item  If the solution $(f_\eps, g_\eps>0)$ of \eqref{syst:ode}-\eqref{syst:bc} exists, then $E_\eps$ has exactly two minimisers that are given by $(f_\eps, \pm g_\eps)$ via \eqref{rad-sol}, while the critical point of $E_\eps$ corresponding to $(\tilde f_\eps, \tilde g_\eps\equiv0)$  via \eqref{rad-sol} is unstable.
\end{enumerate}
\end{proposition}

\begin{proposition}\label{prop:dicho}
Under the hypotheses of Theorem \ref{cor:symmetry}, there exists some $\eps_k > 0$ such that a solution $(f_\eps, g_\eps)$ of \eqref{syst:ode}-\eqref{syst:bc} with $g_\eps > 0$ exists if and only if $\eps \in (0,\eps_k)$. Moreover, $\eps_k=\eps_{-k}$ for every $k\geq 1$ and the sequence $(\eps_{k})_{k\geq 1}$ is increasing.
\end{proposition}

\begin{proof}[Proof of  Proposition~\ref{prop:symmetry}.] Let us fix $\eps>0$. By Theorem \ref{thm:main_inv} there exists a minimiser $u_\eps\in C^1\cap L^\infty(\Omega)$ of $E_\eps$ over $\mcA$. Using the transformation \eqref{transf_n} we can assume that $u_\eps\cdot e_3\geq 0$ in $\Omega$. By \eqref{Eq:Phi>or=0}, we know that: 
\begin{itemize}
\item either $u_\eps \cdot e_3>0$ in $\Omega$ and by Theorem~\ref{thm:main}, $E_\eps$ has exactly two minimisers, namely $u_\eps$ and $\mathcal{I} u_\eps$ where $\mathcal{I}$ is the reflection matrix about the $(x_1, x_2)$-plane cointaining $u_{bd}(\partial \Omega)$;
\item or $u_\eps \cdot e_3 \equiv 0$ in $\Omega$ and $u_\eps$ is the unique minimiser of $E_\eps$ (as point $2$ in Theorem \ref{thm:main_inv} does not hold).
\end{itemize}
Now, we consider the rotation
$$R_k(\theta)=\left(\begin{array}{ccc}
\cos k\theta & - \sin k\theta & 0\\
\sin k\theta & \cos k \theta & 0\\
0 & 0 & 1
\end{array}\right) \quad \textrm{for every $\theta \in [0, 2\pi)$}$$
 and define  the $\bS^{1}$-group action:
\be
\label{rot_sym}
u_{\eps,\theta} (x) = R_k(\theta)^{-1} u_\eps\bigg( \Pi\big(R_1(\theta)\cdot (x,0) \big)\bigg), \quad x\in \Omega,
\ee
where $\Pi:\R^3\to \R^2$ is the projection $\Pi(x_1,x_2,x_3)=(x_1,x_2)$. Then $u_{\eps,\theta}$ is also a minimiser of $E_\eps$ over $\mcA$. Thanks to the above discussion on the uniqueness of $u_\eps$, it follows that 
\[
u_\eps \equiv u_{\eps,\theta} \text{ for all } \theta \in [0,2\pi)
\]
(because $u_{\eps, \theta}\cdot e_3$ and $u_\eps\cdot e_3$ have the same sign).
This invariance implies the existence of three scalar functions $f_\eps, \hat f_\eps, g_\eps\in C^1(0,1)$ such that
$$u_\eps(r\cos \varphi, r\sin\varphi)=f_\eps(r)(\cos k\varphi, \sin k\varphi)+\hat f_\eps(r) (-\sin k\varphi, \cos k\varphi)+g_\eps(r) e_3,$$
for every $r\in (0,1)$ and $\varphi\in [0, 2\pi)$ (see e.g. proof of Proposition 2.1 in \cite{INSZ_AnnIHP}). Moreover, the Euler-Lagrange equation \eqref{E-L} of $u_\eps$ combined with the boundary condition $u_\eps=u_{bd}$ on $\partial \Omega$ imply that $\hat f_\eps\equiv 0$ in $(0,1)$ (i.e., $u_\eps$ is represented as in \eqref{rad-sol}) and the couple 
$(f_\eps, g_\eps)$ satisfies the system \eqref{syst:ode}-\eqref{syst:bc} (see e.g. the proof of Proposition 2.3 in \cite{INSZ_AnnIHP}).

Clearly, $g_\eps\geq 0$ (by the assumption at the beginning of the proof). We want to prove that $f_\eps >0$ in $(0,1)$. First, note that ${u^\circ}\in \mcA$ corresponding to $(|f_\eps|, g_\eps)$ via \eqref{rad-sol} has the same energy as $u_\eps$, i.e., $E_\eps({u^\circ})=E_\eps(u_\eps)$, 
so by the uniqueness of minimisers (see the beginning of the proof) we deduce that $f_\eps=|f_\eps|\geq 0$ in $(0,1)$. Moreover, by Step 2 in the 
proof of Theorem \ref{thm:main_inv} we know that $|u_\eps|\leq 1$ yielding $f_\eps^2+g_\eps^2\leq 1$. In particular, by \eqref{ass:W}, 
we deduce that the right hand side in \eqref{syst:ode} is non-negative. By the strong maximum principle applied to $f_\eps\in H^1(B^2)$ satisfying 
$-\Delta f_\eps+\frac{k^2}{r^2} f_\eps\geq 0$ in $B^2$ and $f_\eps=1$ on $\partial B^2$, we conclude that $f_\eps>0$ in $(0,1)$. 
Also, the strong maximum principle applied to  $g_\eps\in H^1(B^2)$ satisfying
$-\Delta g_\eps\geq 0$ in $B^2$ and $g_\eps\geq 0$ in $B^2$ implies that either $g_\eps>0$ in $B^2$, or $g_\eps\equiv 0$ in $B^2$.

It remains to prove the stated uniqueness of solutions of the system \eqref{syst:ode}-\eqref{syst:bc} satisfying $g_\eps\geq 0$ in $(0,1)$.
\smallskip

\noindent {\it Case 1. } If $(f_\eps, g_\eps > 0)$ is a $C^1$ solution of \eqref{syst:ode}-\eqref{syst:bc}, then the map $u_\eps$ defined by \eqref{rad-sol} with the ``$+$" sign is a critical point of $E_\eps$ with $u_\eps \cdot e_3 = g_\eps > 0$ (as one can easily verify \eqref{E-L}) 
and $f_\eps, g_\eps\in L^\infty$ implies that $u_\eps\in L^\infty$. Thus, by Theorem \ref{thm:main}, $u_\eps$ is the unique minimiser of $E_\eps$ in the set
\[
\mcA_+ = \{u \in \mcA: u \cdot e_3 \geq 0 \text{ a.e. in } \Omega\}.
\]
This proves that there is at most one solution of \eqref{syst:ode}-\eqref{syst:bc} satisfying $g_\eps > 0$.

\smallskip

\noindent {\it Case 2. } If $(\tilde f_\eps > 0, \tilde g_\eps \equiv 0)$ solves  \eqref{syst:ode}-\eqref{syst:bc}, then $\tilde f_\eps$ satisfies
\be
\label{nou_ecuat}
\left\{\begin{array}{l}
-\tilde f_\eps'' - \frac{1}{r} \tilde f_\eps' + \frac{k^2}{r^2} \tilde f_\eps = \frac{1}{\eps^2} \tilde f_\eps \,W'(1-\tilde f_\eps^2 ),\\
\tilde f_\eps(0) = 0, \tilde f_\eps(1) = 1,
\end{array}\right.
\ee
which is known to enjoy existence and uniqueness for every $\eps>0$ (see e.g. \cite{tang, farina-guedda, Hervex2, ODE_INSZ}). Moreover, if the solution 
$(f_\eps, g_\eps>0)$ exists (yielding uniqueness - up to the reflection $\cal I$ - of the minimiser $u_\eps$ corresponding to $(f_\eps, g_\eps>0)$ via \eqref{rad-sol}), we will show that $\tilde u_\eps$ corresponding to $(\tilde f_\eps, \tilde g_\eps\equiv 0)$ via \eqref{rad-sol} is unstable for $E_\eps$ over $\mcA$. For that, assume by contradiction that $\tilde u_\eps$ is stable (in the direction $e_3$). As $\tilde u_\eps$ is ``non-escaping", then Proposition \ref{Lem:InPlane+Stab=>Min} would imply that $\tilde u_\eps$ should be the unique minimiser of $E_\eps$ over $\mcA$ which is a contradiction with the existence of $u_\eps$.
\end{proof}

\begin{proof}[Proof of Proposition~\ref{prop:dicho}] Let $I$ denote the set of $\eps > 0$ such that a solution $(f_\eps, g_\eps)$ of \eqref{syst:ode}-\eqref{syst:bc} satisfying $g_\eps > 0$ exists and let $J = (0,\infty) \setminus I$.
Recall that $(\tilde f_\eps, \tilde g_\eps)$ is the unique solution of \eqref{syst:ode}-\eqref{syst:bc} satisfying $\tilde g_\eps \equiv 0$. By Proposition \ref{prop:symmetry}, $(\tilde f_\eps, \tilde g_\eps)$ is a stable critical point of $E_\eps$ if and only if $\eps \in J$. By the definition of stability of the solution $(\tilde f_\eps, \tilde g_\eps)$ (see Proposition~\ref{Lem:InPlane+Stab=>Min}) and the continuity of the map $\eps\in (0, \infty)\mapsto \tilde f_\eps\in C^0([0,1])$ (see e.g.  \cite{ODE_INSZ}), it follows that $J$ is closed and $I$ is open.

\medskip

\noindent {\it Step 1. We show that there is some $\tilde \eps_k> 0$ such that $(0,\tilde \eps_k) \subset I$.}
Indeed, we note that, as $\eps \rightarrow 0$, minimisers of $E_\eps$ over $\mcA$ converge strongly in $H^2(\Omega)$ to the following $\mathbb{S}^2$-valued minimising harmonic maps with boundary values given by \eqref{Eq:ubdDegk} (see e.g. \cite{BBH}):
$$(r, \varphi)\in (0,1)\times [0, 2\pi)\mapsto \left(\frac{2r^k}{1+r^{2k}} \cos (k\varphi), \frac{2r^k}{1+r^{2k}}\sin(k\varphi), \pm\frac{1-r^{2k}}{1+r^{2k}}\right)\in \bS^2.$$
In view of Proposition \ref{prop:symmetry}, the conclusion of Step 1 follows.

\medskip

\noindent {\it Step 2. We show that there is some $\hat \eps_k> \tilde \eps_k$ such that $[\hat \eps_k, \infty) \nsubseteq I$.}
Indeed, if $\eps>0$ is large enough, then the energy $E_\eps$ is strictly convex (see Remark \ref{rem:convex}) and thus, the unique critical point of $E_\eps$ over $\mcA$ is given by the solution $(\tilde f_\eps >0, \tilde g_\eps\equiv 0)$ of the system \eqref{syst:ode}-\eqref{syst:bc} via \eqref{rad-sol}. The conclusion of Step 2 follows.

\medskip

\noindent {\it Step 3. We prove the desired dichotomy with respect to the critical value 
\[
\eps_k = \sup \{\eps>0\, : \, (0,\eps) \subset I\} \in [\tilde \eps_k, \hat \eps_k].
\]
}
Indeed, we start by noting that $\eps_k \in J$
because $I$ is open. Stability then implies that
\begin{equation}
\int_{\Omega} \Big[|\nabla \varphi|^2 - \frac{1}{\eps_k^2} W'(1 - |\tilde f_{\eps_k}|^2)\,\varphi^2\Big]\,dx \geq 0 \quad \text{ for all } \varphi \in H^1_0(\Omega).
	\label{Eq:26IV18-E1}
\end{equation}
To conclude, we only need to show that, for every $\eps > \eps_k$,
\begin{equation}
\int_{\Omega} \Big[|\nabla \varphi|^2 - \frac{1}{\eps^2} W'(1 - |\tilde f_{\eps}|^2)\,\varphi^2\Big]\,dx \geq 0\quad \text{ for all } \varphi \in H^1_0(\Omega).
	\label{Eq:26IV18-E2}
\end{equation}
(This would mean that $(\tilde f_\eps,\tilde g_\eps)$ is stable and so $\eps \in J$ for all $\eps > \eps_k$.)
Define
\[
\bar f_\eps(x) =  \tilde f_\eps(\eps\,x) \text{ for } x \in \Omega_\eps = \{|x| < \eps^{-1}\}.
\]
Then, by \eqref{Eq:26IV18-E1}, we have
\[
\int_{\Omega_{\eps_k}} \Big[|\nabla \varphi|^2 -  W'(1 - |\bar f_{\eps_k}|^2)\,\varphi^2\Big]\,dx \geq 0 \text{ for all } \varphi \in H^1_0(\Omega_{\eps_k}).
\]
By the comparison principle \cite[Proposition 3.5]{ODE_INSZ}, we have $1\geq \bar f_\eps \geq \bar f_{\eps_k}\geq 0$ in $\Omega_\eps$ for $\eps > \eps_k$, which implies that
\[
W'(1 - |\bar f_{\eps}|^2) \le W'(1 - |\bar f_{\eps_k}|^2) \text{ in } \Omega_\eps.
\]
Extending $\bar f_\eps$ to $\Omega_{\eps_k}$ by setting $\bar f_\eps = 1$ in $\Omega_{\eps_k} \setminus \Omega_{\eps}$, since $W'(0)=0$, we thus obtain
\[
\int_{\Omega_{\eps_k}} \Big[|\nabla \varphi|^2 -  W'(1 - |\bar f_{\eps}|^2)\,\varphi^2\Big]\,dx \geq 0 \text{ for all } \varphi \in H^1_0(\Omega_{\eps_k}).
\]
This implies the stability inequality \eqref{Eq:26IV18-E2} and concludes Step 3.

\medskip

\noindent {\it  Step 4. We prove that there exists a positive function $\varphi_k \in H_0^1(\Omega)$ such that}
\[
\int_{\Omega} \Big[|\nabla \varphi_k|^2 - \frac{1}{\eps_k^2} W'(1 - |\tilde f_{\eps_k}|^2)\,\varphi_k^2\Big]\,dx = 0 .
\]
Indeed, for $\eps>0$ and $\tilde f_\eps$ the solution of \eqref{nou_ecuat} for a fixed $k\in \Z\setminus\{0\}$, let
\[
\lambda_{1,k}(\eps) := \inf_{\varphi \in H^1_0(\Omega), \|\varphi\|_{L^2(\Omega)} = 1}\int_{\Omega} \Big[|\nabla \varphi|^2 - \frac{1}{\eps^2} W'(1 - |\tilde f_{\eps}|^2)\,\varphi^2\Big]\,dx.
\]
As $W'(1 - |\tilde f_{\eps}|^2)\in L^\infty(\Omega)$, standard arguments show that the infimum is attained by some positive and smooth $\varphi_{\eps,k} \in H^1_0(\Omega)$ which satisfies 
\[
-\Delta \varphi_{\eps,k} - \frac1{\eps^2}W'(1 - |\tilde f_{\eps}|^2)\,\varphi_{\eps,k} = \lambda_{1,k}(\eps)\varphi_{\eps,k} \quad \text{ in } \Omega.
\]
Furthermore, by elliptic estimates, $\tilde f_\eps$ and $\varphi_{\eps,k}$ depend continuously on $\eps$. On the other hand, by the dichotomy, $\lambda_{1,k}(\eps) \geq 0$ for all $\eps \geq \eps_k$ and $\lambda_{1,k}(\eps) < 0$ for $\eps < \eps_k$. Hence $\lambda_{1,k}(\eps_k) = 0$. Letting $\varphi_k = \varphi_{\eps_k,k}$ yields Step 4.

\medskip

\noindent {\it  Step 5. We have $\eps_k=\eps_{-k}$ for every $k\geq 1$ and the sequence $(\eps_{k})_{k\geq 1}$ is increasing.} Indeed, by the transformation $u=(u_1,u_2, u_3)\mapsto (u_1,-u_2, u_3)$, the minimisation problem corresponding to the boundary data $u_{bd}$ with winding number $k$ is equivalent with the one of winding number $-k$. Therefore, $\eps_k=\eps_{-k}$ for every $k\geq 1$. Let us now prove that $\eps_{k-1}< \eps_k$ for every $k\geq 2$. To make it clearer as we are now working with different winding numbers, let us now rename $\tilde f_\eps$ to $\tilde f_{\eps,k}$.

Now, note that $\tilde f_{\eps_k,k-1}$ is the solution of the problem \eqref{nou_ecuat} for $\eps=\eps_k$ and the winding number $k-1$, i.e. 
\be
\label{nou_ecuatXX}
\left\{\begin{array}{l}
-f'' - \frac{1}{r}  f' + \frac{(k-1)^2}{r^2} f = \frac{1}{\eps_k^2}f \,W'(1-f^2 ),\\
 f(0) = 0, f(1) = 1.
\end{array}\right.
\ee
On the other hand, $\tilde f_{\eps_k,k}$ is a subsolution to \eqref{nou_ecuatXX} with $\lim_{r \rightarrow 0} \frac{\tilde f_{\eps_k,k}}{r^{k-1}} = 0$. The comparison principle \cite[Proposition 3.5]{ODE_INSZ} applied for the problem \eqref{nou_ecuatXX} implies that $0 < \tilde f_{\eps_{k},k} < \tilde f_{\eps_k,k-1} < 1$ in $(0,1)$.
This yields $W'(1 - |\tilde f_{\eps_{k},k}|^2) > W'(1 - |\tilde f_{\eps_k,k-1}|^2)$ in $(0,1)$. Since by Step 3, $(\tilde f_{\eps_{k},k},0)$ is stable for $E_{\eps_k}$ (see \eqref{Eq:26IV18-E1}), we deduce that $(\tilde f_{\eps_k,k-1},0)$ is also stable for $E_{\eps_k}$. As $\eps_{k-1}$ is the smallest $\eps$ for which $(\tilde f_{\eps,k-1},0)$ is stable for $E_\eps$, we deduce that $\eps_{k}\geq \eps_{k-1}$. On the other hand, if $\eps_k = \eps_{k-1}$, then, by Step 4,
\begin{align*}
0 
	&= \int_{\Omega} \Big[|\nabla \varphi_{k-1}|^2 - \frac{1}{\eps_{k-1}^2} W'(1 - |\tilde f_{\eps_{k-1},k-1}|^2)\,\varphi_{k-1}^2\Big]\,dx \\
	&= \int_{\Omega} \Big[|\nabla \varphi_{k-1}|^2 - \frac{1}{\eps_{k}^2} W'(1 - |\tilde f_{\eps_{k},k-1}|^2)\,\varphi_{k-1}^2\Big]\,dx \\
	&> \int_{\Omega} \Big[|\nabla \varphi_{k-1}|^2 - \frac{1}{\eps_{k}^2} W'(1 - |\tilde f_{\eps_{k},k}|^2)\,\varphi_{k-1}^2\Big]\,dx ,
\end{align*}
which contradicts the stability of $(\tilde f_{\eps_k,k},0)$ for $E_{\eps_k}$. We conclude that $\eps_k > \eps_{k-1}$.
\end{proof}

\section{Harmonic map problem}
\label{sec:HMP}

In this section, we consider $\Sphere^{n-1}$-valued maps. Recall that a map $u \in H^1(\Omega,{\mathbb S}^{n-1})$ is called a weakly harmonic map in $\Omega$ if it satisfies in the weak sense the harmonic map equation
\[
-\Delta u  = |\nabla u|^2 u \text{ in } \Omega.
\]
This is the Euler-Lagrange equation for critical points of the Dirichlet energy
\be\label{def:energ*} 
E(u)=\int_\Omega\frac{1}{2}| \nabla u|^2\, dx.
\ee 
A weakly harmonic map is called a minimising harmonic map in $\Omega$ if there holds
\[
E(u) \leq E(w) \text{ for all } w \in H^1(\Omega,{\mathbb S}^{n-1}) \text{ such that } w|_{\partial \Omega} = u|_{\partial\Omega} .
\]
We start by proving the analogue of Theorem \ref{thm:main} for weakly harmonic maps that are positive in a direction $e\in \bS^{n-1}$
inside $\Omega$.

\begin{theorem}\label{prop:HMP}
Let $m \geq 1$, $n \geq 2$, $\Omega\subset\R^m$ be a bounded domain with smooth boundary, and $u_{bd} \in H^{1/2}(\partial\Omega, \mathbb{S}^{n-1})$ be a given boundary data. Assume that 
$u \in \mcA \cap  H^1(\Omega,{\mathbb S}^{n-1})$ is a weakly harmonic map satisfying
\be\label{ass:Phi+}
u \cdot e>0\textrm{ a.e. in }\Omega
\ee 
in a (fixed) direction $e\in \bS^{n-1}$.
Then $u$ is a minimising harmonic map and we have  the following dichotomy:
\begin{enumerate}
\item If there exists some Lebesgue point $x_0\in\partial\Omega$ of $u_{bd}$ such that $u_{bd}(x_0)\cdot e>0$ then $u$ is the unique minimising harmonic map in $\mcA \cap  H^1(\Omega,{\mathbb S}^{n-1})$.
\item If $u_{bd}(x)\cdot e=0$ for $\mathcal{H}^{m-1}$-a.e. $x\in \partial\Omega$, then all minimising harmonic maps in $\mcA \cap  H^1(\Omega,{\mathbb S}^{n-1})$ are given by $Ru$, where $R\in O(n)$ is an orthogonal transformation of $\R^n$ satisfying $Rx=x$ for all $x\in  Span \, u_{bd}(\partial\Omega)$.
\end{enumerate}
\end{theorem}

\begin{proof}[Proof of Theorem \ref{prop:HMP}] Up to a rotation, we can assume that $e=e_n$ (as in \eqref{en}). It is clear that $\Phi=u \cdot e_n>0$ satisfies 
\be
\label{har}
-\Delta \Phi=|\nabla u|^2\Phi \ \hbox{  in } \ \Omega.
\ee
We consider the perturbation $u+v$ with $v\in H^1_0(\Omega,\R^n)$ and $|u+v|=1$ a.e. in $\Omega$ (in particular, $|v|\leq 2$ in $\Omega$). Then 
\be\label{const:uv}
 2u\cdot v+|v|^2=0 \quad \textrm{ a.e. in $\Omega$. }
\ee  Using the harmonic map equation, we obtain
$$
2\int_\Omega \nabla u\cdot \nabla v=2\int_\Omega |\nabla u|^2 u\cdot v\,dx\stackrel{\eqref{const:uv}}{=}-\int_\Omega |\nabla u|^2 |v|^2\,dx,
$$
which leads to the second variation $\cal D$ of $E$ at $u$:
\be
\label{expd}
\int_\Omega |\nabla (u+v)|^2\,dx-\int_\Omega |\nabla u|^2\,dx =\int_\Omega|\nabla v|^2- |\nabla u|^2|v|^2\,dx=:{\cal D}(v).
\ee
(For simplicity, we still called  $\cal D$ the second variation of $E$ at $u$ even if we do not ask that $v$ is orthogonal to $u$).
To show that $u$ is minimising, we show that ${\cal D}(v)\ge 0$ for all $v\in H^1_0\cap L^\infty(\Omega;\RR^n)$ (note that this is a class larger than what we need, as we do not require that $v$ satisfy the pointwise constraint \eqref{const:uv}).
To this end we take an arbitrary map $\bar v\in C_c^\infty(\Omega; \RR^n)$ and decompose it as  $\bar v = \Phi \psi$ in $\Omega$. 
By \eqref{ass:Phi+} and \eqref{har}, we note that $\Phi$ is a superharmonic function (i.e., $-\Delta \Phi\geq 0$ in $\Omega$) that belongs to $H^1(\Omega)$. Thus, the weak Harnack inequality (see e.g. \cite[Theorem 8.18]{GT}) implies that on the support $\omega$ of $v$ (that is compact in $\Omega$), we have $\Phi\geq \delta>0$ for some $\delta$ (which may depend on $\omega$). 
Thus, one obtains $\psi\in H^1_0\cap L^\infty(\Omega;\RR^n)$ and then, for $1 \leq j \leq n$, integration by parts yields:
\begin{align*}
\int_\Omega |\nabla u|^2 \bar v_j^2\,dx
	&= \int_\Omega  |\nabla u|^2 \Phi \,\Phi\,\psi_j^2\,dx\\
	&\stackrel{\eqref{har}}{=} \int_\Omega \nabla \Phi \cdot \nabla (\Phi\,\psi_j^2)\,dx\\
	&= \int_\Omega \Big[-\Phi^2 |\nabla \psi_j|^2 + |\nabla\bar v_j|^2\Big]\,dx.
\end{align*}
It follows that 
$$
{\cal D}(\bar v) = \int_\Omega \Phi^2 |\nabla \frac{\bar v}{\Phi}|^2\,dx \ge 0 \text{ for all }\bar v\in C_c^\infty(\Omega;\RR^n).
$$ 
It is well known that for every $v\in L^\infty\cap H^1_0(\Omega;\RR^n)$, there exists a sequence $v_k\in C_c^\infty(\Omega;\RR^n)$  such that $v_k\to v$ and $\nabla v_k\to \nabla v$ in $L^2$ and a.e. in $\Omega$ and $|v_k|\leq \|v\|_{L^\infty}+1$ in $\Omega$. In particular, by dominated convergence theorem, we have 
${\cal D}(v_k)\to {\cal D}(v)$ 
thanks to \eqref{expd}. Thus, we deduce that for every compact $\omega\subset \Omega$,
$${\cal D}(v)=\lim_{k\to \infty} {\cal D}(v_k)\geq \liminf_{k\to \infty} \int_\omega \Phi^2 |\nabla \frac{v_k}\Phi|^2\,dx\geq \int_\omega \Phi^2 |\nabla \frac{v}\Phi|^2\,dx\geq 0$$ for all 
$v\in H^1_0\cap L^\infty(\Omega;\RR^n)$, 
where we used Fatou's lemma.
In particular, $u$ is a minimising harmonic map by \eqref{expd}. We can now argue as in the Step 3 of the proof of Theorem~\ref{thm:main} to obtain the rest of the result. We omit the details.
\end{proof}

\begin{proof}[Proof of Theorem \ref{prop:HMP_inv}] Up to a rotation, we may assume that $e_n=e$ with $u_{bd}$ satisfying the non-negativity assumption \eqref{ass:bd} in direction $e_n$. First we note that $\mcA\cap H^1(\Omega; \bS^{n-1})\neq \emptyset$.
Indeed, one can consider the harmonic extension $u_0$ of $u_{bd}$ inside $\Omega$ implying that $u_0(\Omega)\subset \overline{B^n}_+:=\overline{B^{n}}\cap \{x_n\geq 0\}$; 
then set $u$ to be the retraction of $u_0$ onto the hemisphere $\bS^{n-1}_+$ using the pole $-e_n$, namely $u = \Psi \circ u_0$, where $\Psi:\overline{B^n}_+\to \bS^{n-1}_+$ is defined as
\[
\Psi(x',x_n) = (tx', -1 + t(x_n + 1)) \text{ with } t = \frac{2(x_n+1)}{|x'|^2 + (x_n+ 1)^2}, \, |x'|^2+x_n^2\leq 1 \text{ and }x_n \geq 0.
\]
Therefore, $u\in \mcA\cap H^1(\Omega; \bS^{n-1})$. By the direct method of the calculus of variation, we deduce the existence of a minimising harmonic map inside $\mcA\cap H^1(\Omega; \bS^{n-1})$. We claim that any minimising harmonic map $v$ in $\mcA \cap  H^1(\Omega,{\mathbb S}^{n-1})$ satisfies 
\be
\label{cl}
\text{ either } |v_n| > 0 \text{ or } v_n \equiv 0 \text{ in } \Omega.
\ee
To see this, first note that, in view of  \eqref{ass:bd},  $\hat v = (v_1, \ldots, v_{n-1}, |v_n|)$ is also a minimising harmonic map in $\mcA \cap  H^1(\Omega,{\mathbb S}^{n-1})$. Then, by the harmonic map equation, $|v_n| \in H^1(\Omega)$ is super-harmonic in $\Omega$. Thus, if the (essential) infimum of $|v_n|$ on some ball $B$ with $\bar B\subset \Omega$ is zero, then the weak Harnack inequality (see e.g. \cite[Theorem 8.18]{GT}) implies that $|v_n|$ must be identically zero in $B$; moreover, by a continuation argument on every path passing through $B$, we conclude that $|v_n|$ vanishes in $\Omega$  (as $\Omega$ is an open connected set, so path-connected). This implies the claim \eqref{cl}.
The rest of the argument is similar to Step 4 in the proof of Theorem~\ref{thm:main_inv}. We omit the details.
\end{proof}

To conclude this section, we note that the same methods yield the following analogue of Proposition \ref{Lem:InPlane+Stab=>Min} for harmonic maps.

\begin{prop}\label{Lem:HPInPlane+Stab=>Min}
Let $m \geq 1$, $n \geq 2$, $\Omega\subset\R^m$ be a bounded domain with smooth boundary, $e\in \bS^{n-1}$. 
If $u \in H^1(\Omega, \Sphere^{n-1}\cap e^\perp)$ is a weakly harmonic map and $u$ is stable in the direction $e$, i.e.
\be
\label{numar_nou}
\frac{d^2}{dt^2}\big|_{t=0}E\big(\frac{u+t\f e}{|u+t\f e|} \big)=\int_\Omega \Big[|\nabla \varphi|^2 - |\nabla u|^2\,\varphi^2\Big]\,dx \geq 0 \text{ for all } \varphi \in H_0^1(\Omega;\mathbb{R}),
\ee
then $u$ is a minimising harmonic map. Moreover, if $u(\Omega)\subset Span\, u(\partial \Omega)$, then $u$ is the unique minimising harmonic map within its boundary data $u\big|_{\partial \Omega}$.
\end{prop}

\appendix
\section{Appendix}
\label{sec:d0}

In the introduction we discussed the ``escaping" and ``non-escaping" phenomena of minimisers of $E_\eps$ and $E$ considering specific examples in Theorem \ref{cor:symmetry} and Example \ref{examp}. In both cases the boundary data  carried a {\it non-zero topological degree} (as a map $u_{bd}: \partial \Omega \to u_{bd}(\partial\Omega)$), in addition to the assumption that  the co-dimension of $Span \, u_{bd}(\partial \Omega)$ in the target space $\R^n$ is nonzero. These examples illustrated the effect of ``escaping" and ``non-escaping" phenomena on minimisers, in particular, the relation between ``escape'' and smoothness of minimising harmonic maps. 

In this section we would like to give a simple example of ``escape" and ``non-escape" phenomena for minimisers with a boundary data carrying a {\it zero topological degree} on $\partial \Omega$. In particular, we would like to point out that in this case there exist smooth and ``non-escaping" minimising harmonic maps.
More precisely, let $\Omega\subset \R^2$ be a bounded smooth simply connected domain, $n\geq 3$ and for some fixed $a \in \RR$, we consider the ``horizontal" harmonic map 
\be
\label{aleg}
u_{a}(x_1,x_2)=(\cos(ax_1), \sin(ax_1),0,\dots,0)\in \R^n \quad \textrm{for every } x\in \R^2. 
\ee
Restricting $u_a$ to the boundary $\partial \Omega$ we define $u_{bd}:\partial \Omega\to \bS^1=\bS^1\times\{0_{\R^{n-2}}\}$ (i.e., $u_{bd} \equiv u_a \big|_{\partial \Omega}$).  It is clear that $u_{bd}$ carries a zero winding number on $\partial \Omega$ and can be smoothly lifted by $\f_{bd}(x)=ax_1$ on $\partial \Omega$. We will show that with the boundary data $u_{bd}$ the appearance of ``escaping" phenomenon for minimisers of harmonic map problem and Ginzburg-Landau problem strongly depends on the value $|a|$.

\medskip 

\nd {\bf Harmonic map problem}. 

\begin{prop}
Let $u_{bd}$ be defined as above and $\lambda_1(\Omega)$ be the first eigenvalue of $(-\Delta)$ on $\Omega$ with zero Dirichlet data. Then 
\begin{enumerate}
\item {\bf ``Non-escaping"}: For $a^2\leq\lambda_1(\Omega)$, the map $u_a$ defined in \eqref{aleg} is the unique minimising harmonic map in $\mcA\cap H^1(\Omega, \bS^{n-1})$.

\item {\bf ``Escaping"}: For $a^2> \lambda_1(\Omega)$ every minimising harmonic map $u\in \mcA\cap H^1(\Omega, \bS^{n-1})$ ``escapes" from the horizontal plane $\R^2\times \{0_{\R^{n-2}}\}$ and $u$ is unique up to orthogonal transformations $R\in O(n)$ fixing every point of the horizontal plane. Moreover, the harmonic map $u_{a}$ is unstable. 
\end{enumerate} 
\end{prop}

\begin{proof} We first note that for any $a \in \R$ the set $\mcA\cap H^1(\Omega, \bS^{n-1})\neq \emptyset$ since $u_{a}$ belongs to this set and by the direct methods there exists a minimising harmonic map $u\in \mcA\cap H^1(\Omega, \bS^{n-1})$. Moreover, we observe that $u_a=(\cos \f_a, \sin \f_a, 0, \dots, 0)$ with $\f_a(x)=ax_1$ for every $x\in \R^2$; in particular, $\f_a$ is harmonic and hence $u_a$ is a weakly $\bS^{n-1}$-valued harmonic map, i.e., 
$$
-\Delta u_a = |\nabla u_a|^2 u_a \hbox{ in } \Omega.
$$
Taking any $v \in H_0^1(\Omega, \R^n)$ such that $|u_a + v|=1$ and following the proof of the Theorem~\ref{prop:HMP} we obtain
\be\label{dv}
\int_\Omega |\nabla (u_a+v)|^2\,dx-\int_\Omega |\nabla u_a|^2\,dx =\int_\Omega|\nabla v|^2- \underbrace{|\nabla u_a|^2}_{=|\nabla \f_a|^2=a^2}|v|^2\,dx=:{\cal D}(v).
\ee
\vskip 0.1cm

\noindent{\it 1.} If $a^2 \leq \lambda_1(\Omega)$, then Poincar\'e's inequality and \eqref{dv} yield ${\cal D}(v) \geq 0$ for every $v\in H^1_0(\Omega, \R^n)$; in particular, $u_a$ is a minimising harmonic map. Moreover, by Theorem~\ref{prop:HMP_inv}, we deduce that $u_a$ is the unique minimising harmonic map in $\mcA\cap H^1(\Omega, \bS^{n-1})$.

\vskip 0.1cm

\noindent{\it 2.} Let us take $a^2 >\lambda_1(\Omega)$.  Assume by contradiction that $u$ does not ``escape" from the horizontal plane, i.e., 
$u\in H^1(\Omega, \bS^{1}\times\{0_{\R^{n-2}}\})$. In particular, $u$ is a minimiser of $E$ restricted to ``horizontal" configurations 
$\mcA\cap H^1(\Omega, \bS^{1}\times\{0_{\R^{n-2}}\})$. Since $\Omega$ is a simply connected domain, then any map $w \in \mcA\cap H^1(\Omega, \bS^{1}\times\{0_{\R^{n-2}}\})$ can be represented as $w=(\cos \f, \sin \f, 0, \dots, 0)$ with the lifting $\varphi \in H^1(\Omega)$ and $\f=\f_{bd}$ on $\partial \Omega$. In particular, if $\f$ is the lifting of $u$, then $\f$ is harmonic (because $u$ is a harmonic map); as $\f=\f_{bd}$ on $\partial \Omega$, one obtains that $\f(x)=\f_a(x)=a x_1$ in $\Omega$. Therefore $u=u_{a}$ in $\Omega$, i.e.,  $u_a$ is a minimising harmonic map. We take $v \in H_0^1(\Omega)$ and compute the second variation of $E$ at the point $u_a$ in the direction $v e_3$ (that is orthogonal to $u_a$) to obtain
$$
\frac{d^2}{dt^2}\big|_{t=0}E\big(\frac{u_a+t v e_3}{|u_a+t v e_3|} \big) = \int_\Omega |\nabla v|^2-\underbrace{|\nabla u_a|^2}_{=a^2}|v|^2\,dx.
$$
Since $a^2 >\lambda_1(\Omega)$ we can find $v\in H^1_0(\Omega)$ making the second variation negative (in the direction $v e_3$) which contradicts the minimality of $u_a$ over $\mcA\cap H^1(\Omega, \bS^{n-1})$. This proves in particular that the (unique) ``horizontal" critical point $u_{a}$ of $E$ is unstable (when vertical directions are admissible in the target space $\R^n$, i.e., $n\geq 3$).
The uniqueness of minimising harmonic maps is then given by Theorem \ref{prop:HMP_inv}.
\end{proof}

\medskip

\nd {\bf Ginzburg-Landau problem}. We can transfer the above ideas to the case of the Ginzburg-Landau problem. In particular, we have:

\vskip 0.1cm

$\bullet$ {\it ``Escaping" phenomenon for $a^2>\lambda_1(\Omega)$ and small enough $\eps\leq \eps_a$}. Similar to  minimising harmonic maps, in the case $a^2>\lambda_1(\Omega)$ the ``escaping" phenomenon happens for minimisers of $E_\eps$ over $\mcA$ within the boundary data $u_{bd}$, provided 
$\eps\leq \eps_a$ with $\eps_a$ depending on $a$. The proof uses the fact that as $\eps \to 0$ the minimisers $u_\eps$ of $E_\eps$ over $\mcA$ converge in $H^1(\Omega)$ 
to minimising harmonic maps $u$, the key point is $H^1$ regularity of the minimising harmonic maps $u$ 
(see e.g. \cite{BBH}). The uniqueness holds up to isometries $R\in O(n)$ keeping every point of the horizontal plane fixed. 

$\bullet$ {\it ``Non-escaping" phenomenon for any $a \in \R$ and large enough $\eps\geq \eps_0$}. Using Remark~\ref{rem:convex} we know that $E_\eps$ is strictly convex over $\mcA$ for large enough $\eps\geq \eps_0$ (here $\eps_0$ depends only on $W'(1)$ and $\lambda_1(\Omega)$ and is {\it independent} of $a$). Therefore there exists a unique critical point (that is the global minimiser) $u_\eps$ of $E_\eps$ over $\mcA$. 
Moreover, $u_\eps$ is ``non-escaping" for every $\eps\geq \eps_0$ and for every $a\in \R$ since by Theorem \ref{thm:main_inv}  the existence of an ``escaping" minimiser $u_\eps$ of $E_\eps$ over $\mcA$ would break the uniqueness (as the reflected map ${\cal I} u_\eps\neq u_\eps$ w.r.t. the horizontal plane would also be a minimiser of $E_\eps$ over $\mcA$).

\section*{Acknowledgment.} R.I. acknowledges partial support by the ANR project ANR-14-CE25-0009-01. V.S. acknowledges  support by the EPSRC grant EP/K02390X/1 and the Leverhulme
grant RPG-2014-226. R.I and V.S. thank for hospitality  and
support of the Basque Center for Applied  Mathematics  (BCAM) where a part of this work was done. A.Z. was partially supported by a Grant of the Romanian National
Authority for Scientific Research and Innovation, CNCS-UEFISCDI, project number
PN-II-RU-TE-2014-4-0657; by the Basque Government through the BERC
2014-2017 program; and by the Spanish Ministry of Economy and Competitiveness
MINECO: BCAM Severo Ochoa accreditation SEV-2013-0323.

\def\cprime{$'$}

\end{document}